\documentclass[a4paper, 10pt, twoside, notitlepage]{amsart}

\usepackage[utf8]{inputenc}
\usepackage{color}
\usepackage{amsmath}
\usepackage{amssymb}
\usepackage{amsthm}
\usepackage{graphicx}
\usepackage{esint}
\usepackage[colorlinks=true,linkcolor=blue]{hyperref}
\usepackage{verbatim}

\theoremstyle{plain}
\newtheorem{thm}{Theorem}
\newtheorem{prop}{Proposition}[section]
\newtheorem{lem}[prop]{Lemma}
\newtheorem{cor}[prop]{Corollary}

\newtheorem{defi}[prop]{Definition}
\newtheorem{rmk}[prop]{Remark}

\newcommand {\R} {\mathbb{R}} 
 \newcommand {\N} {\mathbb{N}}
 
\newcommand {\dH} {d\mathcal{H}^n}
\newcommand {\p} {\partial}

\newcommand {\va} {\varphi}

\newcommand {\D} {\Delta}

\newcommand {\supp} {\text{supp}}

\newcommand{\LL}{\tilde{L}^2}

\DeclareMathOperator{\argmin}{argmin}

\DeclareMathOperator {\dist} {dist}
\DeclareMathOperator {\Ree} {Re}
\DeclareMathOperator {\Imm} {Im}

\begin{document}

\title[The Thin Obstacle Problem with Hölder Coefficients]{Optimal Regularity for the Thin Obstacle Problem with $C^{0,\alpha}$ Coefficients}
\author{Angkana R\"uland }
\author{Wenhui Shi}

\address{
Mathematical Institute of the University of Oxford, Andrew Wiles Building, Radcliffe Observatory Quarter, Woodstock Road, OX2 6GG Oxford, United Kingdom }
\email{ruland@maths.ox.ac.uk}

\address{
Department of Mathematics, University of Coimbra, Apartado 3008,
3001-501 Coimbra, Portugal }
\email{wshi@mat.uc.pt}

\begin{abstract}
In this article we study solutions to the (interior) thin obstacle problem under low
regularity assumptions on the coefficients, the obstacle and the underlying manifold. Combining the linearization method of
Andersson \cite{An16}
and the epiperimetric inequality from \cite{FS16}, \cite{GPSVG15}, we prove the optimal
$C^{1,\min\{\alpha,1/2\}}$ regularity of solutions in the presence of $C^{0,\alpha}$ coefficients
$a^{ij}$ and $C^{1,\alpha}$ obstacles $\phi$. Moreover we investigate
the regularity of the regular free boundary and show that it has the structure of a
$C^{1,\gamma}$ manifold for some $\gamma \in (0,1)$.
\end{abstract}

\subjclass[2010]{Primary 35R35}

\keywords{Variable coefficient Signorini problem, epiperimetric inequality, optimal regularity, free boundary regularity, low coefficient regularity, linearization technique}

\maketitle

\section{Introduction}
\subsection{Problem and  Setting}
In this article we prove the optimal regularity of solutions to the \emph{(interior) thin
obstacle problem} in a low regularity framework involving $C^{0,\alpha}$ coefficients and
$C^{1,\alpha}$, $\alpha\in (0,1)$, obstacles. Moreover, we address the regularity of
the \emph{regular} free boundary.\\

More precisely we consider the following set-up:
Let $\Omega\subset \R^{n+1}$ be a bounded domain with sufficiently smooth boundary.
Suppose that $\mathcal{M}$ is an $n$-dimensional submanifold, which decomposes $\Omega $
into two components $\Omega_+$ and $\Omega_-$. Let $\phi\in C^1(\overline{\Omega})$
be given with $\phi < 0$ on $\mathcal{M}\cap \p\Omega$.
We are interested in solutions $w\in \mathcal{K}_{\phi}$ to the variational inequality
\begin{align}
\label{eq:var_in}
\int_{\Omega}a^{ij}\p_iw \p_j(v-w) dx \geq 0 \text{ for any } v\in \mathcal{K}_{\phi},
\end{align}
where
\begin{align*}
\mathcal{K}_\phi=\{v\in W^{1,2}_0(\Omega): v\geq \phi \text{ on } \mathcal{M}\}.
\end{align*}
We assume that the coefficients $a^{ij}$ are bounded, measurable and uniformly elliptic.
Problems of this type for instance arise in the modeling of osmosis \cite{DL76}. The obstacle
is ``thin'' in the sense that it is only active on a co-dimension one manifold.\\

The variational inequality (\ref{eq:var_in}) possess a unique solution
$w\in \mathcal{K}_{\phi}$ \cite{LS69}.
Furthermore, $w$ can be represented as
\begin{align*}
w(x)=\int_\Omega G(x,y)d\mu
\end{align*}
where $G(x,y)$ is the Green's function to the operator $\p_ia^{ij}\p_j$
with respect to $\Omega$, and $\mu$ is a Radon measure associated with $w$ with $\supp\mu\subset \mathcal{M}\cap\{w=\phi\}$. By Evan's lemma (c.f. Theorem 2.1 in
\cite{LS69}), if $\mathcal{M}$ is a continuous manifold, then the solution $w$ is
continuous in $\Omega$. \\
This together with the representation formula implies that in the relatively open set
$\Omega\setminus (\mathcal{M}\cap \{w=\phi\})$, the solution $w$ to (\ref{eq:var_in})
satisfies the equation
\begin{equation}
\label{eq:off_fb}
\p_ia^{ij}\p_jw=0.
\end{equation}
Moreover, it obeys a weak form of \emph{complementary} or \emph{Signorini} boundary
conditions on $\mathcal{M}$, if $\mathcal{M}$ is for instance $C^1$ regular.
To make this precise, we define the following operators
\begin{equation}
\label{eq:normal}
\begin{split}
\p_{\nu_+}u(x)&:=\lim_{y\rightarrow x,\ y\in \Omega_+}\nu_+(y)\cdot A(y)\nabla u(y),\\
 \p_{\nu_-}u(x)&:=\lim_{y\rightarrow x,\ y\in \Omega_-}\nu_-(y)\cdot A(y)\nabla u(y),
\end{split}
\end{equation}
for functions $u \in C^1(B_\delta(x)\cap \overline{\Omega_\pm})$ for some $\delta>0$ and for coefficients $A(y):=(a^{ij}(y)) \in C(B_\delta(x), \R^{(n+1)\times (n+1)})$.
Here $\nu_+(x)$ and $\nu_-(x)$ are the outer unit normals to
$\Omega_+$ and $\Omega_-$ at $x\in \mathcal{M}$, respectively.
The solution $w\in \mathcal{K}_{\phi}$ to the variational
inequality (\ref{eq:var_in}) then satisfies a weak form of a jump condition for
the co-normal boundary derivative:
With the notation from (\ref{eq:normal}) it solves
$\p_{\nu^+}w+\p_{\nu^-}w\geq 0$ as a distribution on $\mathcal{M}\cap \Omega$, i.e. for all
$\eta\in W^{1,2}_0(\Omega)$ with $\eta\geq 0 \text{ on } \mathcal{M}\cap \Omega$ it holds
\begin{align*}
\int_{\mathcal{M}\cap \Omega} (\p_{\nu^+}w+\p_{\nu^-}w)\eta dx
:=\int_{\Omega} a^{ij}\p_j w \p_i \eta dx\geq 0.
\end{align*}
Furthermore, invoking \eqref{eq:off_fb}, we deduce that
\begin{align}\label{eq:sig}
\supp(\p_{\nu^+}w+\p_{\nu^-}w)\subset \mathcal{M}\cap \{w=\phi\}.
\end{align}

In the sequel, we seek to prove the local $C^{1,\alpha}$ regularity of $w$ under
suitable regularity assumptions on the coefficients $a^{ij}$,
the manifold $\mathcal{M}$ and the obstacle $\phi$.
To this end, we assume that for some $\alpha\in (0,1)$,
\begin{itemize}
\item the coefficients $a^{ij}$ are $C^{0,\alpha}$ regular,
\item $\mathcal{M}$ is $C^{1,\alpha}$ regular,
\item and $\phi:\mathcal{M}\rightarrow \R$ is $C^{1,\alpha}$ regular.
\end{itemize}
Proving a \emph{local} regularity result, we may further,
without loss of generality, assume that $\Omega=B_1$ and
$\mathcal{M}=B'_1:= \{x\in \R^{n}\times \{0\}: |x|\leq 1\}$.
Moreover, passing from $w$ to $\tilde{w}:= \eta(w - \phi)$, where $\eta\in C^\infty_0(B_1)$ is a cut-off function such that $\{w=\phi\}\Subset \{\eta=1\}$, and with slight abuse of notation
dropping the tildas in the notation, it suffices to study variational
inequalities of the form
\begin{equation}
\label{eq:vari_0}
\begin{split}
&\int_{B_1}a^{ij}\p_iw\p_j(v-w) dx+\int_{B_1}g^j\p_j(v-w) dx\geq 0 \\
&\text{ for any } v\in \{v\in W^{1,2}_0(B_1): v\geq 0 \text{ on } B'_1\}.
\end{split}
\end{equation}
with $a^{ij}\in C^{0,\alpha}(B_1)$ and $g^i \in C^{0,\alpha}(B_1)$.
Without loss of generality, we will furthermore assume the following normalization condition
\begin{equation}
\label{N}
\tag{N}
\begin{split}
& a^{ij}(0)=\delta^{ij},\  g^i(0)=0 \mbox{ and }\\
& a^{i, n+1}(x',0)= 0,\ i\in\{1,\cdots, n\},\text{ for all }(x',0)\in B_1'.
 \end{split}
\end{equation}
The first two equations can always be achieved by rotating, rescaling and by substracting a constant. For the off-diagonal normalization we rely on a change of coordinates introduced by Uraltseva (page 1183 in \cite{U89}). Here we note that the change of coordinates in \cite{U89} is stated for $W^{1,p}$, $p>n+1$, coefficients, but it is not hard to see from its proof that it also holds for $C^{0,\alpha}$ coefficients with $\alpha\in (0,1)$.\\

In the sequel, we study the regularity of $w$ in the set-up of (\ref{eq:vari_0}) and (N).
In addition, we investigate the \emph{free boundary} $\Gamma_w:= \partial \{x\in B_1': w(x)>0\}$,
which divides $B_1'$ into the \emph{contact set} $\Lambda_w:=\{x\in B'_1: w=0\}$ and the
\emph{non-conincidence set} $\Omega_w:=\{x\in B'_1: w>0\}$.

\subsection{Main results and ideas}
In this article we study the optimal regularity of the solution and the $C^{1,\gamma}$
regularity of the \emph{regular} free boundary. Our first main result in this context asserts
that also in the framework of coefficients and obstacles of low regularity, a similar optimal regularity result as in the isotropic situation is valid:

\begin{thm}\label{thm:opt_solution}
Let $w$ be a solution to \eqref{eq:vari_0} with $a^{ij}, g^i\in C^{0,\alpha}(B_1)$,
$\alpha\in (0,1)$, satisfying the normalization condition (N).
\begin{itemize}
\item[(i)] If $\alpha\in (0,1/2)\cup (1/2,1)$,
then $w\in C^{1,\min\{\alpha,1/2\}}(B_{1/2}^\pm)$. Moreover,
there exists a constant $C=C(n,\alpha,\|a^{ij}\|_{C^{0,\alpha}}, \|g^i\|_{C^{0,\alpha}})$
such that
\begin{align*}
\|w\|_{C^{1,\min\{\alpha,1/2\}}(B_{1/2}^\pm)}\leq C\|w\|_{L^2(B_1)}.
\end{align*}
\item[(ii)] If $\alpha=1/2$, then $w\in C^{1,\beta}(B_{1/2}^\pm)$ for any
$\beta\in (0,1/2)$. Moreover, there exists
$C=C(\beta, n,\|a^{ij}\|_{C^{0,\alpha}}, \|g^i\|_{C^{0,\alpha}})$ such that
\begin{align*}
\|w\|_{C^{1,\alpha}(B_{1/2}^\pm)}\leq C\|w\|_{L^2(B_1)}.
\end{align*}
\end{itemize}
\end{thm}

Let us comment on this result: If $\alpha \in (0,1/2)$ and the coefficients and
inhomogeneities are $C^{0,\alpha}$ regular, the only limitations on the regularity of the
solution is given by the elliptic regularity results away from the free boundary. If
$\alpha>1/2$, the typical obstacle type behavior kicks in and the solution is intrinsically
limited in its regularity. Due to the presence of the model solution (for the isotropic problem,
in which $a^{ij}=\delta^{ij}$)
\begin{align*}
 h_{3/2}(x):= \Ree(x_1 + i |x_{n+1}|)^{3/2},
\end{align*}
the limitation to $C^{1,1/2}$ regularity is optimal. The situation $\alpha = \frac{1}{2}$
is interesting, as this exactly corresponds to the borderline case, in which the first non-analytic
``eigenfunction'' at the free boundary plays an important role limiting the regularity. As explained in Remark
\ref{rmk:3/2} we can improve the estimate given in Theorem \ref{thm:opt_solution} (ii) slightly,
by showing an only logarithmic loss with respect to the $C^{1,1/2}$ growth behavior, which
holds for $\alpha>1/2$. In Remark \ref{rmk:3/2} we present an example illustrating that this logarithmic loss is indeed optimal.\\

Our method of proof for Theorem \ref{thm:opt_solution} relies on a linearization technique, which was developed
by Andersson \cite{An16} in the context of the Lam\'e problem.
The central ideas are a perturbative analysis of
the problem around the homogeneous, constant coefficient global problem and the
exploitation of the
precise knowledge of its lowest order global solutions (in the form of Liouville type
results, c.f. Lemmas \ref{lem:class} and \ref{lem:global}).
In the main part of the argument, we thus consider the projection of solutions to the expected
global profiles at each scale, to derive a precise convergence rate towards these. This is achieved by a compactness argument, which investigates the renormalized error terms after each projection, and by the study of the limiting problem it leads to. 
Here a major ingredient is the epiperimetric inequality (c.f. Theorem \ref{thm:epi}), which was derived in \cite{FS16} and \cite{GPSVG15} following the approach of Weiss \cite{Weiss99} to the classical obstacle problem.
This energy based result allows us to view our variable coefficient
problem as a \emph{perturbation of the constant coefficient, homogeneous situation}. In particular,
it allows us to conclude a decay rate for the associated Weiss energy (c.f. Corollary \ref{cor:perturb_epi}),
which in turn yields the necessary compactness for Andersson's linearization technique. \\

Following these ideas, we prove Theorem~\ref{thm:opt_solution} in two steps. In the first
step (c.f. Section \ref{sec:almost_opt}), we develop an \emph{almost optimal} interior
regularity result for solutions
(up to an arbitrarily small $\epsilon$ loss) by identifying a convergence
rate of the solution to the leading order profile, $a(x_{n+1})_+ + b(x_{n+1})_-$, which is of linear growth
(c.f. Lemma \ref{lem:growth_a} and Proposition \ref{prop:Hoelder}).
In the second step (c.f. Section \ref{sec:opt_reg}), we establish a quantitative
convergence rate of $w-(a(x_{n+1})_+ + b(x_{n+1})_-)$ to
the next possible global profile $c \Ree(Qx'+ix_{n+1})^{3/2}$, $Q\in SO(n)$, $c\in \R_+$
(c.f. Lemma \ref{lem:negative}). To this end, we project the solution at each scale to the orthogonal complement of the above $3/2$ global profile and show that the projected functions decay substantially faster (c.f. Lemma~\ref{lem:negative}).
Here we use the epiperimetric inequality (c.f. Theorem \ref{thm:epi}) to obtain a geometric decay of the Weiss energy along the projected functions. This implies sufficient compactness for the linearization technique and hence enables us to obtain the convergence rate of the blow-up families towards the global profiles.\\

As the second major result of the article and as a by-product of the decay rate towards the
$3/2$-homogeneous global profiles,
we further deduce the local $C^{1,\gamma}$ regularity of the \emph{regular} free boundary,
if $\alpha>1/2$. Here the regular free boundary consists of all points with vanishing order
(in the sense of the theorem below) strictly less than $1+\alpha$.

\begin{thm}
Let $w$ be a solution to \eqref{eq:vari_0} with $a^{ij}, g^i\in C^{0,\alpha}(B_1)$,
$\alpha\in (1/2,1)$, satisfying the normalization condition (N).
Assume that $x_0\in \Gamma_w\cap B_{1/2}$ with
\begin{align*}
\liminf_{r\rightarrow 0_+}\frac{ \ln (\|w-\p_{n+1}w(x_0)x_{n+1}\|_{\tilde{L}^2(B_r)})}{\ln r}<1+\alpha.
\end{align*}
Then there exist $r_0\in (0,1/2)$ and $\gamma\in (0,1)$ depending on
$w,\alpha,n, \|a^{ij}\|_{C^{0,\alpha}}$ and $\|g^i\|_{C^{0,\alpha}}$,
such that $\Gamma_w\cap B_{r_0}(x_0)$ is a $C^{1,\gamma}$ regular $(n-1)$-dimensional
manifold.
\end{thm}

Finally, we also explain how the methods, which we use, apply to the \emph{boundary} thin obstacle problem (or Signorini problem) by a slight variation of our proof. Here the coefficient regularity can be reduced substantially (c.f. Remark \ref{rmk:boundary}).

\subsection{Background}
The study of the regularity properties of the thin obstacle problem has a long tradition,
dating back to the works of Lewy \cite{Le72} and Richardson \cite{Ri78}, who established the optimal
regularity result for the two-dimensional problem by means of complex
analysis techniques. Subsequently, the full higher dimensional problem was studied in
different generalities by various authors, including the articles of Caffarelli \cite{Ca79},
Kinderlehrer \cite{Ki81} and Uraltseva \cite{Ura87}. Of these, the work of Uraltseva implied the most
general results: Working with $W^{1,p}$ coefficients and $L^p$ inhomogeneities with $p>n+1$, she showed the $C^{1,\alpha}$ regularity of solutions to the thin obstacle
problem, however with an undetermined value $\alpha\in(0,1)$. In their seminal articles Athansopoulos and Caffarelli \cite{AC06} and Athansopoulos,
Caffarelli and Salsa \cite{ACS08} for the first time identified and proved the optimal $C^{1,1/2}$
regularity in the \emph{isotropic, constant coefficient} situation \emph{without inhomoneity}.
Here the authors introduced Almgren's frequency function as a crucial new tool yielding
sufficient compactness to carry out a blow-up analysis of the problem. \\
While the \emph{constant} coefficient (and the \emph{high} regularity) situation was studied thoroughly
subsequently \cite{G09}, \cite{CSS}, \cite{GP09}, \cite{FS16} (c.f. also \cite{PSU} for further
references), the \emph{variable} coefficient situation with \emph{low coefficient regularity}
was only addressed more recently \cite{GSVG14}, \cite{GPSVG15}, \cite{KRS14}, \cite{KRSI}. In this context all the cited articles either relied on a very careful analysis of Almgren's frequency
formula or Carleman estimates as a replacement of this. As these tools are however
limited to the setting of $W^{1,p}$ with $p=\infty$ and $p>n+1$ coefficient regularity,
respectively, the situation of only H\"older continuous coefficients had been out of their
scope.\\
A robust, new line of thought was however introduced by
Andersson \cite{An}, \cite{An16} in the context of the
Lam\'e system, where he implemented a linearization technique. Our strategy of proof
combines his ideas with energetic arguments -- in the form of the Weiss energy
and the epiperimetric inequality -- to derive optimal regularity results for the scalar thin
obstacle problem with only H\"older continuous coefficients.

\subsection{Outline}
The remainder of the article is organized as follows: In Section \ref{sec:almost_opt} we study
the almost optimal regularity properties of solutions to the variable coefficient thin obstacle
problem. Here we argue in two steps, first establishing the almost Lipschitz regularity
(c.f. Lemma \ref{lem:almost_linear}) as the main conclusion of Section
\ref{subsec:almost_linear} and then proceed in Section \ref{subsec:almost_opt} to the
full \emph{almost optimal regularity} result (c.f. Proposition \ref{prop:Hoelder}).
With this preparation, in Section \ref{sec:opt_reg}, which is the central part of our argument, we investigate the \emph{optimal regularity} of solutions
(c.f. Theorem \ref{thm:Hoelder_opt}). Again this discussion is split into two parts: We
first recall and discuss the Weiss energy and its relation to the epiperimetric inequality
(c.f. Theorem \ref{thm:epi}). In particular, we show that even for our variable coefficient
situation
it is possible to derive a decay rate for the Weiss energy by
relying on the epiperimetric inequality from \cite{FS16} (c.f. Corollary \ref{cor:negative}).
With this at hand, we then combine Andersson's compactness method with the quantitative
decay of the Weiss energy to deduce a decay rate of solutions to the variable coefficient
thin obstacle problem towards the $3/2$-homogeneous profile (c.f. Lemma \ref{lem:negative}).
This then allows us to derive the desired optimal regularity result in Theorem
\ref{thm:Hoelder_opt}. In Section \ref{sec:freebdr} we then further exploit the decay estimate
from Lemma \ref{lem:negative} to infer the regularity of the regular free boundary
(c.f. Theorem \ref{thm:freebound}). Finally, we conclude the article by providing the proof
of Lemma \ref{lem:class} in the Appendix, Section \ref{sec:app}.

\section{Almost Optimal Regularity}\label{sec:almost_opt}

In this section we prove the \emph{almost} optimal regularity of solutions to the thin obstacle problem (c.f. Proposition \ref{prop:Hoelder}). The argument for this relies on the compactness method of proving regularity, as introduced by Andersson \cite{An16}, and
the precise understanding of global solutions to the constant coefficient thin obstacle problem. In order to implement this strategy, in Section \ref{subsec:almost_linear} we first establish some sufficient (i.e. sublinear) initial regularity, for which we rely on energy estimates as in \cite{AU} (c.f. Lemma \ref{lem:AU}) in combination with compactness ideas as in Andersson (c.f. Lemma \ref{lem:almost_linear}). In the second part of this section (c.f. Section \ref{subsec:almost_opt}) we then carry out Andersson's compactness argument (c.f. Lemma \ref{lem:growth_a}). \\

\subsection{Almost linear regularity}\label{subsec:almost_linear}

In this section we establish some initial regularity for solutions to \eqref{eq:vari_0}. These results, which rely on ideas of Arkhipova and Uraltseva \cite{AU} and Andersson \cite{An16}, allow us to rely
on pointwise estimates for solutions to the variable coefficient thin obstacle problem in the later sections. As a major tool we also prove a Liouville type result (c.f. Lemma \ref{lem:global}). In the following sections it will play a central role.\\

We begin by recalling the $C^{0,\beta}$
regularity estimates due to Arkhipova and Uraltseva \cite{AU}:

\begin{lem}
\label{lem:AU}
 Let $w$ be a solution to the variable coefficient thin obstacle problem \eqref{eq:vari_0} in $B_{1}$ with $a^{ij},g^i\in C^{0,\alpha}(B_1)$ for some $\alpha\in (0,1)$. Further assume that the condition (N) is satisfied.
 Then, for some $\beta \in (0,1)$ we have that $w\in C^{0,\beta}(B_{1/2})$ and
 \begin{align*}
  \|w\|_{C^{0,\beta}(B_{1/2})}
  \leq C\|w\|_{L^2(B_1)},\quad C=C(\|a^{ij}\|_{C^{0,\alpha}}, \|g^{i}\|_{C^{0,\alpha}}, n,\alpha).
 \end{align*}
\end{lem}

\begin{proof}
We only sketch the proof, as it can be found in the articles of Arkhipova and Uraltseva,
in particular in \cite{AU}. We consider the following penalized problem
\begin{equation}
\label{eq:penalized}
 \begin{split}
  \p_i a^{ij} \p_ j w^{\epsilon}& =-\p_i g^i \mbox{ in } B_1,\\
  \p_{\nu_+}w^{\epsilon}+\p_{\nu_-}w^{\epsilon}
  &= -\beta_{\epsilon}(w^{\epsilon}) \mbox{ on } B_1',\\
  w^{\epsilon} &= 0 \mbox{ on } \partial B_1,
 \end{split}
 \end{equation}
 where $\epsilon>0$ is a small parameter, $\beta_{\epsilon}:\R \rightarrow (-\infty,0)$, $\beta_{\epsilon}'\geq 0$,
 $\beta_{\epsilon}(t)=0$ for $t\geq 0$ and $\beta_{\epsilon}(t)=\epsilon + t/\epsilon$ for
 $t \leq -2\epsilon^2$. We work with (\ref{eq:penalized}) in its weak formulation
\begin{align}\label{eq:energy_equ}
 \int\limits_{B_1} a^{ij} \p_i w^{\epsilon} \p_j \eta dx =
 - \int\limits_{B_1'} \beta_{\epsilon}(w^{\epsilon}) \eta dx'
 - \int\limits_{B_1} g^i \p_i \eta dx,\quad  \forall \eta\in C^\infty_0(B_1).
\end{align}
It is classical that an energy solution to this problem exists.
Indeed, since $\beta'_\epsilon\geq 0$, the operator
\begin{align*}
A_\epsilon&: W^{1,2}_0(B_1)\rightarrow W^{-1,2}(B_1),\\
\langle A_\epsilon u,\eta\rangle &=\int_{B_1}a^{ij}\p_iu\p_j\eta dx
+\int_{B'_1}\beta_\epsilon(u)\eta dx',
\end{align*}
is monotone.
Moreover, it is not hard to check that it is bounded and coercive,
i.e. $\langle A_\epsilon u, u\rangle \geq \lambda \|u\|_{W^{1,2}_0(B_1)}$.
Thus existence follows from the Brouwer fixed point theorem (c.f. \cite{L76}).
\\

With this at hand, we next derive the uniform (in $\epsilon$) regularity of the solution $w^\epsilon$.
First using the weak formulation,
it is immediate to deduce the uniform boundedness of $w^{\epsilon}$ in $W^{1,2}(B_1)$ (the
weak formulation of the equation gives the homogeneous $W^{1,2}$ bound; the boundary data
and an application of Poincar\'e's inequality the corresponding $L^2$ estimate).
Next we consider the test functions
$\eta_{\pm}:=(w^{\epsilon} \mp k)_+ \zeta^2$, where $k\in \R_+$ and $\zeta$
denotes a cut-off function supported in $B_{\rho}(0)$. This yields that for
$A_{k,\rho}:= \{ x\in B_{\rho}: w^{\epsilon}(x) \geq k  \}$ and $\rho \in (0,1/2)$
\begin{align*}
 \int\limits_{A_{k,\rho}}|\nabla w^{\epsilon}|^2 \zeta^2 dx
 \leq C \int\limits_{A_{k,\rho}}(w^{\epsilon}\mp k)_+^2 |\nabla \zeta|^2 dx
 + C \|g^i\|_{L^{\infty}}|A_{k,\rho}|.
\end{align*}
Invoking Theorem 5.2 in \cite{LU} implies that
$w^\epsilon\in L^\infty(B_1)$ with $\|w^\epsilon\|_{L^\infty(B_1)}\leq C$,
where $C=C(n,\lambda, \Lambda, \|g^i\|_{L^\infty(B_1)}, \|w^\epsilon\|_{L^2(B_1)})$.
Then an application of Theorem 6.1 in \cite{LU} results in
$w^{\epsilon}\in C^{0,\beta}(B_{1/2})$ for some $\beta\in (0,1)$ with a norm,
which only depends on $\|w^{\epsilon}\|_{L^2(B_1)}$. As
$\|w^{\epsilon}\|_{L^2(B_1)} \leq C \|g^i\|_{L^2(B_1)}$, this is
in particular independent of the parameter $\epsilon$.\\

Finally, we prove the convergence of $w^\epsilon$ to $w$. This relies on two ingredients:
\begin{itemize}
\item[(i)] By the previous bounds, we infer that there exists a function $\hat{w}\in W^{1,2}(B_1)$ such that
$w^{\epsilon}\rightarrow \hat{w}$ weakly in $W^{1,2}(B_1)$, and strongly both in $L^2(B_1)$ and $L^2(\partial B_1)$.
We claim that $\hat{w}\geq 0$ for $\mathcal{H}^{n}$ a.e. $x\in B_1'$.
\item[(ii)] The function $\hat{w}$ satisfies the variational inequality \eqref{eq:vari_0}.
\end{itemize}
To obtain (i), we insert the function $\eta= w^{\epsilon}\zeta^2$, with $\zeta \in C^{\infty}_0(B_1)$ being a cut-off function, into \eqref{eq:energy_equ}. This implies that
\begin{align*}
\int\limits_{B_1'} \beta_{\epsilon}(w^{\epsilon}) w^{\epsilon} \zeta^2 dx \leq C \left(\|w^{\epsilon}\|_{W^{1,2}(B_1)}^2 + \|g\|_{L^2(B_1)}^2 \right)
\leq C \|g\|_{L^2(B_1)}^2.
\end{align*}
Hence, for all $\delta>0$ the definition of $\beta_{\epsilon}$ entails that
\begin{align*}
\mathcal{H}^{n}\left( \{w^{\epsilon}(\cdot,0)<-\delta\}\cap \{\zeta=1\}\right)
\delta (\delta- \epsilon^2) \leq C \epsilon,
\end{align*}
which yields the claim of (i).\\
To deduce (ii), we use the test function $\eta= w^{\epsilon}-v$ in \eqref{eq:energy_equ} for an arbitrary function $v\in W^{1,2}(B_1)$, which satisfies $v= w^{\epsilon}=0$ on $\partial B_1$ and $v\geq 0$ on $B'_1$. Invoking the monotonicity of $\beta_{\epsilon}$ and using that $v\geq 0$ on $B'_1$, we obtain $\beta_\epsilon(w^\epsilon)(w^\epsilon-v)=(\beta_\epsilon(w^\epsilon)-\beta_\epsilon(v))(w^\epsilon-v)\geq 0$. This leads to the inequality
\begin{align*}
\int\limits_{B_1} a^{ij} \p_i w^{\epsilon} \p_j (w^{\epsilon}-v) dx + \int\limits_{B_1} g^i \p_i(w^{\epsilon}-v) dx \leq 0.
\end{align*}
Passing to the limit and using weak lower semicontinuity thus results in
\begin{align*}
\int\limits_{B_1} a^{ij} \p_i \hat{w} \p_j (\hat{w}-v) dx + \int\limits_{B_1} g^i \p_i(\hat{w}-v) dx \leq 0.
\end{align*}
By virtue of the choice of $v$ and the boundary data of $\hat{w}$, this however implies that $\hat{w}$ is a solution to the same variational inequality as $w$. By uniqueness this therefore yields the identity $\hat{w}=w$.
\end{proof}

We next seek to upgrade the $C^{0,\beta}$ estimates for $w$ to its almost Lipschitz regularity. To this end, we rely on a Liouville theorem (c.f. Lemma \ref{lem:global}), which will also play a crucial role in the later sections. In order to prove the Liouville theorem, we rely on a classification result for solutions to the following linear, mixed Dirichlet-Neumann boundary value problem:

\begin{lem}[Linear classification result]
\label{lem:class}
Let $v\in W^{1,2}(B_1)$ be a solution of the following linear mixed boundary value problem:
\begin{equation}
\label{eq:problem}
\begin{split}
 \D v &= 0 \mbox{ in } B_1^\pm,\\
 \p_{n+1}^+ v + \p_{n+1}^- v &= 0 \mbox{ on } B_1'\cap\{x_n > 0\},\\
 v &= 0 \mbox{ on } B_1'\cap \{ x_n \leq 0\}.
 \end{split}
\end{equation}
Then,
\begin{align*}
v(x)= \sum\limits_{j=1}^{\infty} a_j q_j(x),
\end{align*}
where the functions $q_j$ are homogeneous functions of the order $j/2$.\\
We further remark that the functions $q_j$ are orthogonal
with respect to the $L^2(B_1)$, the $L^2(\partial B_1)$ and the $W^{1,2}(B_1)$ inner products.\\
Moreover, if
$v\in W^{2,2}(B_1^\pm)$, we have that $a_1 = 0$.
\end{lem}

The proof of this lemma follows from an analogous result of Andersson's (c.f. \cite{An16}, Lemma 6.2) and is postponed to the
Appendix, c.f. Section \ref{sec:app}. \\
Relying on the previous classification result, we formulate and prove the following central Liouville type
theorem for the thin obstacle problem. The key ingredient here is the Friedland-Hayman inequality (c.f. \cite{FH76}): It allows us to show that global solutions with subquadratic growth at infinity are two-dimensional. Analogous results for the \emph{boundary} thin obstacle problem are known (c.f. Lemma 5.1 in \cite{An16} and also Proposition 9.9 in \cite{PSU}). The proof for the interior thin obstacle problem uses the same idea, but with the modification that we do not have even symmetry about $x_{n+1}$ at hand.

\begin{lem}[Liouville]
\label{lem:global}
Let $w\in W^{1,2}_{loc}(\R^{n+1})$ be a global solution to the
thin obstacle problem for $a^{ij}= \delta^{ij}$, i.e. suppose that for any $R\geq 1$,
\begin{equation}\label{eq:limit_equ}
\begin{split}
&\int_{B_R}\nabla w\cdot \nabla (v-w) dx\geq 0,\\
& \quad \mbox{ for all } v\in \{v\in W^{1,2}(B_R):v\geq 0 \text{ on } B'_R, \ v=w\text{ on } \p B_R\}.
\end{split}
\end{equation}
\begin{itemize}
\item[(i)] Assume that $\sup\limits_{B_R}|w|\leq C R^\beta$ for $R\rightarrow \infty$ and some $\beta \in (0,1)$. Then $w(x)=c$ for some constant $c\geq 0$.
\item[(ii)] Assume that $\sup\limits_{B_R}|w|\leq CR^{\beta}$ for $R\rightarrow \infty$ and some $\beta \in (0,3/2)$.
Then, $w(x)= a |x_{n+1}| + b x_{n+1}$ or $ w(x)= c$ for some constants $a,b,c\in \R$ with $a\leq 0$ and $c\geq 0$.
\item[(iii)] Assume that $\sup\limits_{B_R}|w|\leq CR^{\beta}$ for $R\rightarrow \infty$ and some $\beta \in (0,2)$, and that $0\in \Gamma_w$.
Then, $w(x)= b x_{n+1} + d \Ree(Q x' + i |x_{n+1}|)^{3/2}$ for some constants $b, d\in \R$ and a rotation $Q$.
\end{itemize}
In particular, in both cases (i) and (ii) we have $\Gamma_w =\emptyset$.
\end{lem}

We remark that as formulated the Liouville theorem is much stronger than needed for our first application in Lemma \ref{lem:almost_linear}. We will however use it in many further instances in the sequel and have thus stated the result in its full strength here.

\begin{proof}
By the regularity properties of solutions to the thin obstacle problem with constant coefficients,
we have that
$w\in W^{2,2}_{loc}(\R^{n+1}_\pm)$ and $$\sup _{B_R}|\nabla w|
\leq \frac{C}{R}\sup_{B_{2R}}|w|\leq CR^{\beta-1},\quad R>1.$$
Observe that $(\p_iw)^+$ and $(\p_iw)^-$ with $i\in \{1,\cdots,n\}$ are subharmonic
with sublinear growth at infinity (by the assumption that $\beta<2$). Thus, by the Friedland-Hayman inequality \cite{FH76},
the tangential derivatives $\p_iw$ have a sign, i.e. for each $i\in\{1,\cdots,n\}$ either
$$\p_iw\geq 0 \mbox{ or } \p_iw\leq 0$$
in the whole space $\R^{n+1}$. This implies that $w$ is two-dimensional. \\
We study the corresponding global problem in two dimensions. Without loss of generality we assume that $w(x)=w(x_1,x_{n+1})$.
We distinguish three cases: $\Omega = \R^n \times \{0\}$, $\Lambda = \R^n \times \{0\}$
and $\Gamma_w \neq \emptyset$:
\begin{itemize}
\item[(i)] We assume that $\Omega = \R^n \times \{0\}$. In this case $w$ is a weak solution to the Laplace equation in the whole space. By interior regularity, this implies that it is a classical solution in the whole space. Hence, invoking the growth assumption, we infer that the only possible solutions are the affine ones.
\item[(ii)] Next, we suppose that  $\Lambda = \R^n \times \{0\}$. By an odd reflection of the solution in the upper half space and by invoking the subquadratic growth at infinity, we obtain that the resulting function $w^+(x)$ has to be a linear function in $x_{n+1}$. A similar result holds for the oddly reflected solution of the lower half-space $w^-(x)$. Therefore,
\begin{align*}
w(x)= a(x_{n+1})_+ + b(x_{n+1})_- \mbox{ for } a,b \in \R.
\end{align*}
\item[(iii)] We finally consider the situation, in which $\emptyset \neq \Gamma_w \subset \R^{n} \times \{0\} $. Without loss of generality, we consider the restriction of $w$ and $\Gamma_w$ onto the two-dimensional subspace spanned by the $e_1$ and $e_{n+1}$ directions.
We first claim that $\Gamma_w$ can at most consist of a single point. This is a consequence of $\p_1w\geq 0 $ on $\{x_{n+1}=0\}$ (which implies that $w(x',0)\geq w(\hat x_1,0)>0$ for all $x'>\hat x_1$, $\hat x_1\in \Omega_w=\{(x_1,0): w(x_1,0)>0\}$). Thus the free boundary $\Gamma_w$ consists of only one point, which we assume to be the origin.
By Lemma \ref{lem:class} and Remark \ref{rmk:class} the two-dimensional function $w$ has the form,
\begin{align*}
w(x) = b x_{n+1} + d\Ree(x_1 + i |x_{n+1}|)^{3/2}.
\end{align*}
This yields the representation claimed in (iii).\\
If $\beta \in (0,3/2)$, we in addition deduce that $d=0$. This however is be in contradiction
with the assumption that $\Gamma_w \neq \emptyset$. Thus, with the growth assumptions as in (i), (ii) from above, we infer that the case $\Gamma_w \neq \emptyset$
cannot occur.
\end{itemize}
This concludes the proof.
\end{proof}

With the result of Lemma \ref{lem:AU} and Lemma \ref{lem:global} at hand, we proceed to the almost Lipschitz regularity of solutions to \eqref{eq:vari_0}:

\begin{lem}\label{lem:almost_linear}
Let $w$ be a solution to the thin obstacle problem with $\|w\|_{L^2(B_1)}=1$.
Assume that $a^{ij}\in C^{0,\alpha}(B_1)$ and $g^i\in C^{0,\alpha}(B_1)$ for some
$\alpha\in (0,1)$. Further suppose that the normalization assumption (N) is satisfied.
Then for any $\beta\in (0,1)$ there exists
$C=C(\beta, \|a^{ij}\|_{C^{0,\alpha}}, \|g^i\|_{C^{0,\alpha}},n)>0$ such that
\begin{align*}
\sup_{B_{r}(x_0)}|w|\leq C r^{\beta},\quad \forall r\in (0,1/4), \quad \forall x_0\in B_{1/2}\cap \Gamma_w.
\end{align*}
\end{lem}

\begin{proof}
Without loss of generality we assume that $0\in \Gamma_w$ and first show the estimate at
the origin. The argument for the other free boundary points follows similarly.\\
We assume that the conclusion were wrong: Then there existed $\beta\in (0,1)$ and a
sequence of solutions $w^k$ with $\|w^k\|_{L^2(B_1)}=1$ and $w^k(0)=0$, and radii $r_k$ such that
\begin{align*}
\sup_{B_{r_k}}|w^k|=kr_k^{\beta} \mbox{ and }
\sup_{B_{R}}|w^k|\leq kR^{\beta} \mbox{ for any } R\geq r_k.
\end{align*}
By Lemma~\ref{lem:AU}, $\sup_{B_{r_k}}|w^k|\leq C r_k^{\tilde{\beta}}$
for some $\tilde{\beta}\in (0,1)$ and for some universal constant $C>0$.
This together with the contradiction assumption implies that $r_k\rightarrow 0_+$. \\
Now we consider the blow-up $\tilde{w}^k(x):=\frac{w^k(r_kx)}{kr_k^\beta}$. The choice of $r_k$ yields that
\begin{align}\label{eq:growth_inf}
\sup_{B_{\tilde{R}}}|\tilde{w}^k|\leq \tilde{R}^{\beta}\quad \mbox{ for all } \tilde{R}>1,\text{ and } \sup_{B_1}|\tilde{w}^k|=1.
\end{align}
Furthermore, by Lemma~\ref{lem:AU},
the functions $\tilde{w}^k$ are uniformly bounded in $W^{1,2}_{loc}(\R^{n+1})$ and
$C^{0,\tilde{\beta}}_{loc}(\R^{n+1})$. Thus up to a subsequence,
$\tilde{w}^k\rightarrow w_0$ weakly in $W^{1,2}_{loc}(\R^{n+1})$ and strongly in
$C^{0,\beta}_{loc}(\R^{n+1})$ for all $\beta \in (0,\tilde{\beta})$. We claim that for any $R>0$ the limiting function $w_0$ is a solution to (\ref{eq:limit_equ}). In order to observe this, we note that
the functions $\tilde{w}^k$ solve
\begin{align*}
\int_{B_{1/r_k}}a^{ij}(r_k\cdot)\p_i\tilde{w}^k\p_j\eta dx +\int_{B_{1/r_k}}\frac{r_k^{1-\beta}}{k}g^i(r_kx)\p_i\eta dx
\geq 0\\
\mbox{ for all } \eta\in C^\infty_0(B_{1/r_k}) \mbox{ such that } \eta + \tilde{w}^{k}\geq 0 \mbox{ on } B_{1/r_k}'.
\end{align*}
We consider a test function $\eta$ which is admissible with respect to $w_0$, i.e. $\eta \in C^{\infty}_0(B_R)$ and $w_0 + \eta \geq 0$ on $B'_R$. We claim that it is possible to
find a sequence $\eta_k\in C^\infty_0(B_R)$ such that $\eta_k$ is an admissible test function for
$\tilde{w}^k$ in $B_R$, i.e.
$\eta_k+\tilde{w}_k\geq 0$ on $B'_R$, and moreover satisfies
$\eta_k\rightarrow\eta$ strongly in $W^{1,2}(B_R)$. Indeed, this is a consequence of
the uniform convergence of $\tilde{w}^k$ to $w_0$ in $B_R$: If $\|\tilde{w}^k-w_0\|_{L^{\infty}(B_R)}\leq \epsilon$,
one can for instance consider the functions
$\eta_k(x):= \eta(x) + \epsilon + \va_k(x)$. Here $\va_k$ is a $C^{\infty}$ function
interpolating between
zero and $\epsilon$, chosen such that $\eta_k$ has compact support in $B_R$ and a gradient that is controlled in terms of the quotient
$\epsilon/d$, where $d$ denotes the distance of $\supp(\eta)$ from the boundary of $B_R$.\\
Thus, using the continuity of $a^{ij}$ and $g^i$ and the normalization assumption (N), the limit $w_0$ satisfies
\begin{align*}
\int\limits_{\R^{n+1}} \nabla w_0\cdot \nabla \eta dx\geq 0 \mbox{ for all } \eta\in C^\infty_0(\R^{n+1}) \mbox{ with } \eta + w_0 \geq 0.
\end{align*}
By a density argument this implies that $w_0$ is indeed a solution to \eqref{eq:limit_equ}.\\
By virtue of \eqref{eq:growth_inf} and the locally uniform convergence, we further
infer that for some $\beta\in (0,1)$
$$\sup_{B_R}|w_0|\leq R^\beta\text{ for any } R>1 \text{ and }\sup_{B_1}|w_0|=1.$$
Then however the Liouville type result from Lemma \ref{lem:global} in combination with $w_0(0)=0$
yields that $w_0\equiv 0$ in $\R^{n+1}$. This is a contradiction to our normalization and the strong convergence.
\end{proof}

The almost linear growth estimate at the free boundary together with the standard elliptic estimates away from the free boundary implies the almost Lipschitz regularity of the solution. Since the argument is standard, we do not repeat it here but refer to the proof of Proposition 4.22 in \cite{KRS14} for example.

\subsection{Almost optimal regularity}\label{subsec:almost_opt}
In this section, we implement the compactness argument from \cite{An16} to upgrade the regularity result for our solutions to the variable coefficient thin obstacle problem from $C^{0,\gamma}$ with $\gamma \in(0,1)$ to the almost optimal $C^{1,\beta}$, with $\beta \in(0,\min\{1/2,\alpha\})$, regularity (c.f. Proposition \ref{prop:Hoelder}). As already in the proof of Lemma \ref{lem:almost_linear}, the Liouville theorem from Lemma \ref{lem:global} plays an important role in this.\\

Analyzing solutions of the variable coefficient thin obstacle problem around the free
boundary and expecting that there the solutions to leading order behave
like linear functions, we introduce the following notation:

\begin{defi}[Projections]
\label{defi:proj}
Let
$$\mathcal{P}=\{\ell:\R^{n+1}\rightarrow \R^{n+1}: \ell(x)=a_0x_{n+1},  a_0\in \R\}.$$
We use $\Pr(u,r,x_0)$ to denote the $L^2$ projection of $u\in L^2(B_r(x_0))$ onto $\mathcal{P}$. More precisely, $\Pr(u,r,x_0)\in \mathcal{P}$ satisfies
\begin{align*}
\|u-\Pr(u,r,x_0)\|_{\tilde{L}^2(B_r(x_0))}=\inf_{p\in \mathcal{P}}\|u-p\|_{\tilde{L}^2(B_r(x_0))},
\end{align*}
where here and in the sequel, $\|\cdot\|_{\LL(\Omega)} := \frac{1}{|\Omega|}\|\cdot\|_{L^2(\Omega)}$ for bounded $\Omega \subset \R^{n+1}$.
If $x_0$ is the origin, we abbreviate $\Pr(u,r):=\Pr(u,r,0)$ for simplicity.
\end{defi}

Connecting the possible blow-up solutions and the projections from Definition \ref{defi:proj},
we make the following remark:

\begin{rmk}
\label{rmk:global_proj}
If $u_0(x)=a_0(x_{n+1})_+-a_1(x_{n+1})_-$ for some constants $a_0, a_1$ with $a_0-a_1\leq 0$, then a direct computation yields that
\begin{align*}
\Pr(u_0,r)=\frac{a_0+a_1}{2}x_{n+1} \mbox{ for all }  r\in (0,1),
\end{align*}
and
$
(u_0-\Pr(u_0,r))(x)=\frac{a_0-a_1}{2}|x_{n+1}|.
$
\end{rmk}

In the next three lemmata, we collect properties of the projections, which
will play an important role in our further discussion:

\begin{lem}
\label{lem:growth2}
Let $w\in L^2_{loc}(\R^{n+1})$ and
suppose that for a small constant $\mu\in(0,1)$, and for all $R\geq 1$
\begin{align}
\label{eq:growth2}
\|w-\Pr(w,R)\|_{\tilde{L}^2(B_R)}\leq \mu \|w\|_{\tilde{L}^2(B_R)}.
\end{align}
Then for any $\beta >0$, a sufficiently small choice of $\mu =\mu(\beta)>0$
yields
\begin{align*}
 \|w\|_{\LL(B_R)} \leq R^{1+\beta}\|w\|_{\LL(B_1)} \mbox{ for all } R\geq 2.
\end{align*}
\end{lem}

\begin{rmk}
As the lemma is formulated as a growth result for large $R$,
we emphasize that is is of particular interest for $\beta>0$ being a \emph{small} constant.
\end{rmk}

\begin{proof}
For all $R\geq 1$, our assumption (\ref{eq:growth2}) implies
\begin{align}
\label{eq:aux1}
\frac{1}{1+\mu}\|\Pr(w,R)\|_{\tilde{L}^2(B_R)}\leq \|w\|_{\tilde{L}^2(B_R)}\leq\frac{1}{1-\mu}\|\Pr(w,R)\|_{\tilde{L}^2(B_R)}.
\end{align}
Furthermore,
\begin{align*}
& \mu\|w\|_{\tilde{L}^2(B_{2R})}\geq 2^{(n+1)/2}\|w-\Pr(w,2R)\|_{\tilde{L}^2(B_{R})}\\
&\geq 2^{(n+1)/2}\|\Pr(w,R)-\Pr(w,2R)\|_{\tilde{L}^2(B_{R})}-2^{(n+1)/2}\|w-\Pr(w,R)\|_{\tilde{L}^2(B_{R})}\\
&\geq 2^{(n+1)/2}\|\Pr(w,R)-\Pr(w,2R)\|_{\tilde{L}^2(B_{R})}-\mu 2^{(n+1)/2}\|w\|_{\tilde{L}^2(B_{R})},
\end{align*}
which upon rearrangement results in
\begin{align}\label{eq:iteration}
2^{(n+1)/2}\|\Pr(w,R)-\Pr(w,2R)\|_{\tilde{L}^2(B_{R})}\leq \mu\left(\|w\|_{\tilde{L}^2(B_{2R})}+2^{(n+1)/2}\|w\|_{\tilde{L}^2(B_{R})}\right).
\end{align}
Let $\Pr(w,R)=: a_R x_{n+1}$ for some constant $a_R \in \R$. We note that as $\mu<1$,
we in particular obtain $a_R\neq 0$ and infer the identities
\begin{align*}
\|\Pr(w,R)\|_{\tilde{L}^2(B_R)}&= R|a_R|,\\
\|\Pr(w,R)-\Pr(w,2R)\|_{\tilde{L}^2(B_{R})}&= R|a_R-a_{2R}|.
\end{align*}
Using (\ref{eq:aux1}) and \eqref{eq:iteration}, we hence deduce that
\begin{align}
\label{eq:mu}
2^{(n+1)/2}|a_R-a_{2R}|\leq \frac{\mu}{1-\mu}\left(a_{2R}+2^{(n+1)/2}a_{R}\right).
\end{align}
We define a constant $\sigma_R$ by setting $a_{2R}=:\sigma_R a_R$.
The inequality \eqref{eq:mu} then implies that
\begin{align*}
1-\tau(\mu)\leq \sigma_R \leq 1+\tau(\mu),
\end{align*}
for some $\tau(\mu)$ which can be arbitrarily close to $0$ (independently of $R\geq1$), if $\mu$
is chosen sufficiently small.
Thus,
\begin{align*}
\|w\|_{\tilde{L}^2(B_{2^{k}R})}
&\leq \frac{1}{1-\mu}\|\Pr(w,2^{k}R)\|_{\tilde{L}^2(B_{2^{k}R})}\leq \frac{2^{k}R|a_{2^{k}R}|}{1-\mu}\leq \frac{2^{k}R(1+\tau(\mu))^k|a_{R}|}{1-\mu} \\
&\leq \frac{1+\mu}{1-\mu} 2^{k}(1+\tau(\mu))^k\|w\|_{\tilde{L}^2(B_{R})}.
\end{align*}
Therefore for any $\beta>0$ there exists sufficiently small $\mu=\mu(\beta)>0$ (e.g. one can take $\mu$ such that $\frac{1+\mu}{1-\mu}(1+\tau(\mu))^k\leq 2^{k\beta}$),
such that if \eqref{eq:growth2} is satisfied, then
\begin{align*}
\|w\|_{\tilde{L}^2(B_{R})}\leq  R^{1+\beta}\|w\|_{\LL(B_1)} \mbox{ for all } R\geq 2.
\end{align*}
\end{proof}

As a second key property of the projections, we note the following orthogonality result:

\begin{lem}[Orthogonality]
\label{lem:orthogonal}
Let $w_k:B_1 \rightarrow \R$ be a sequence of functions in $L^2(B_1)$. Assume that
\begin{align*}
 \|w_k - \Pr(w_k,1)\|_{\LL(B_1)} =: \delta_k \rightarrow 0,
\end{align*}
and that
\begin{align*}
 v_k(x):= \frac{w_k - \Pr(w_k,1) }{\delta_k} \rightharpoonup v_0 \mbox{ in } L^2(B_1).
\end{align*}
Then, $\Pr(v_0,1)=0$.
\end{lem}

\begin{proof}
As $\Pr(w_k,1)$ denotes the best linear approximation to $w_k$, we have that for all $k\in \N$
\begin{align*}
 \int\limits_{B_1} x_{n+1}(w_k - \Pr(w_k,1)) dx = 0.
\end{align*}
In particular,
$\int\limits_{B_1} x_{n+1}v_k dx = 0$.
Using the weak convergence of $v_k$ to $v_0$ hence results in
$\int\limits_{B_1} x_{n+1}v_0 dx = 0$.
The claim now follows by exploiting this orthogonality:
$$
 \|v_0 - a x_{n+1} \|_{\LL(B_1)} ^2
 =\|v_0\|_{\LL(B_1)}^2 + a^2\|x_{n+1}\|_{\LL(B_1)}^2 \geq \|v_0\|_{\LL(B_1)}^2.
$$
\end{proof}

\begin{rmk}
\label{rmk:orthogonal}
We stress that the previous lemma and its proof do not build on special properties of the projections onto the space $\mathcal{P}$ of linear polynomials, but is a property that is enjoyed by all ($L^2$) projections onto finite dimensional, closed cones (where the notion of orthogonality is suitably adapted to the corresponding set-up). In Section \ref{sec:opt_reg} we will frequently use this in the context of the space $\mathcal{E}$ (c.f. Definition \ref{defi:proj_3/2}).
\end{rmk}

Finally, as a last auxiliary result before addressing the actual growth estimates for solutions
to the thin obstacle problem, we exploit the almost linear growth (c.f. Lemma \ref{lem:almost_linear}), in order to deduce bounds on the
projections to normalized solutions to the variable coefficient thin obstacle problem:

\begin{lem}
\label{lem:proj}
Let $w:B_1 \rightarrow \R$ be a solution to the variable coefficient
thin obstacle problem with $a^{ij},g^i\in C^{0,\alpha}(B_1)$ for some $\alpha\in (0,1)$, satisfying the normalization condition (N).
Assume that $\|w\|_{\LL(B_1)}=1$ and that $u(0)=0$. Let $r\in (0,1)$ and define
$a_r:= \frac{|\Pr(w,r)|}{|x_{n+1}|}$. Then, for any
$\epsilon \in (0,1)$ and any $r\in(0,1/2)$
\begin{align*}
 a_r \leq C(\epsilon, \|g\|_{C^{0,\alpha}}, \|a^{ij}\|_{C^{0,\alpha}}) r^{-\epsilon}.
\end{align*}
\end{lem}

\begin{proof}
Lemma \ref{lem:almost_linear} yields that $w \in C^{0,\beta}(B_{1/2})$ for all $\beta \in(0,1)$ and that
\begin{align*}
\sup\limits_{B_r} |w|
  \leq Cr^{\beta},\quad C=C(\beta,n,\|a^{ij}\|_{C^{0,\alpha}},\|g^i\|_{C^{0,\alpha}}).
\end{align*}
Using that $|\Pr(w,r)|\leq C \sup\limits_{B_r}|w|$ and choosing $\beta \in (0,1)$
sufficiently close to one, thus implies the desired result.
\end{proof}

With these auxiliary results at hand, we now proceed to a first central growth estimate for
solutions to the thin obstacle problem around the free boundary. This provides the
basis for the almost optimal regularity result of Proposition \ref{prop:Hoelder}:

\begin{lem}
\label{lem:growth_a}
Let $w$ be a solution to the variable coefficient thin obstacle problem with
$a^{ij}, g^i\in C^{0,\alpha}(B_1)$ satisfying the normalization condition (N).
Assume that $\|w\|_{\tilde{L}^2(B_1)}=1$ and that $0\in \Gamma_w$.
Then for all $\beta \in (0,\min\{\alpha,1/2\})$ there exists $C=C(n,\|a^{ij}\|_{C^{0,\alpha}}, \|g^i\|_{C^{0,\alpha}},\beta,\alpha)$
such that for any $r\in (0,1/2)$
\begin{align*}
\left\|w-\Pr(w,r)\right\|_{\tilde{L}^2(B_r)}\leq C r^{1+\beta}.
\end{align*}
\end{lem}

\begin{proof}
We first give an outline of the proof for the convenience of the reader.
Suppose that the statement were wrong.
Then there existed a parameter $\beta \in (0,\alpha)$, a sequence $r_k$ of radii and
a sequence $w_k$ of solutions to the variable coefficient thin obstacle problem with $\|w_k\|_{\tilde{L}^2(B_1)}=1$, $0\in \Gamma_{w_k}$, $a^{ij}_k$ and $g^i_k$ satisfying (N), such that
\begin{align*}
\|w_k - \Pr(w_k,r_k)\|_{\LL(B_{r_k})} \geq k r_k^{1+\beta}.
\end{align*}
Note that by Lemma~\ref{lem:almost_linear} this implies that $r_k \rightarrow 0$. Furthermore, we can choose $r_k>0$ such that
\begin{equation}\label{eq:r_k}
\begin{split}
&\|w_k - \Pr(w_k, r)\|_{\LL(B_r)} \leq k r^{1+\beta} \mbox{ for all } r\in [r_k,3/4],\\
&\|w_k - \Pr(w_k, r_k)\|_{\LL(B_{r_k})} = k r^{1+\beta}_k.
\end{split}
\end{equation}
We consider the normalized error term
\begin{align}\label{eq:v_k}
v_k(x):=\frac{(w_k-\Pr(w_k,r_k))(r_k x)}{\|w_k-\Pr(w_k,r_k)\|_{\tilde{L}^2(B_{r_k})}}
\end{align}
and proceed in three steps to show that our initial contradiction assumption was false.
\\
In Step 1 we show that the functions $v_k$ converge to $v_0$ along a subsequence, where $v_0$ is a global solution to \eqref{eq:limit_equ} with $\|v_0\|_{\tilde{L}^2(B_1)}=1$ and $v_0(x)=a_0 |x_{n+1}|$ for some $a_0<0$. The compactness of $v_k$ (c.f. Step 1a) follows by our choice of $r_k$ and the interior H\"older estimate from Lemma~\ref{lem:proj}. The classification of $v_0$ is proved in Step 1b, where the key is to show that $v_0$ has less than $R^{3/2}$ growth for large $R$.\\
In Steps 2 and 3 we use the fact that $x=0$ is a free boundary point of $w_k$ and the Signorini condition \eqref{eq:sig} to obtain a contradiction. More precisely, in the non-conincidence set $\Omega_{w_k}$ we have $(\p_{\nu_+}+\p_{\nu_-})w_k=0$. This together with Step 1 and an argument of continuity implies that for $\mu>0$ small, there exist $s_k$ and $x_k$ such that the limit of
\begin{align*}
\tilde{v}_k(x):=\frac{(w_k-\Pr(w_k, s_kr_k,x_k))(s_kr_kx+x_k)}{\|w_k-\Pr(w_k,s_kr_k,x_k)\|_{\tilde{L}^2(B_{s_kr_k}(x_k))}}
\end{align*}
has distance $\mu$ to the profile $\gamma|x_{n+1}|$, i.e.
$$\inf_{\gamma\in \R}\|\tilde{v}_0-\gamma|x_{n+1}|\|_{\tilde{L}^2(B_1)}=\mu,\quad \tilde{v}_0=\lim_{k\rightarrow\infty}\tilde{v}_k.$$
 This is however a contradiction, because if $\mu>0$ is sufficiently small, one can apply Lemma~\ref{lem:global} and an orthogonality argument to show that $\tilde{v}_0(x)=b |x_{n+1}|$.  \\

\emph{Step 1: Blow-Up.}
Let $v_k$ be as \eqref{eq:v_k}.
We show that there exists a subsequence $r_k\rightarrow 0$ such that $v_{k}\rightarrow v_0 $ in $L^2_{loc}$,
where $v_0$ is a global solution to the constant coefficient
thin obstacle problem in the sense of the variational inequality (\ref{eq:limit_equ}).
Furthermore, $v_0$ satisfies the growth condition
\begin{align}\label{eq:growth3}
\sup_{B_R}|v_0|\leq C R^{1+\beta}, \quad R\geq 1.
\end{align}
By Lemma~\ref{lem:global} and the orthogonality argument from Lemma \ref{lem:orthogonal}
this then implies that
$$v_0(x)=a_0|x_{n+1}|\text{ for some } a_0<0.$$

\emph{Argument for Step 1:}

\emph{Step 1a: Compactness.}
We note that by the normalization \eqref{eq:r_k}, $v_k$ satisfies $\|v_k\|_{\tilde{L}^2(B_1)}=1$ and $\|v_k\|_{\tilde{L}^2(B_R)}\leq R^{1+\beta}$ for all $R>1$. Furthermore, it solves the variational
inequality
\begin{equation}\label{eq:vari_v_k}
\begin{split}
&\int_{B_{1/r_k}}a^{ij}_k(r_k\cdot)\p_jv_k\p_i(v-v_k) dx
+\int_{B_{1/r_k}} \tilde{g}^i_k \p_i(v-v_k) dx \geq 0,\\
&\text{ for all } v\in \{v\in W^{1,2}(B_{1/r_k}):v\geq 0 \text{ on } B'_{1/r_k},\ v=v_k \text{ on } \p B_{1/r_k}\},
\end{split}
\end{equation}
where
\begin{align*}
\hat{g}^i_k(x)&=\frac{r_kg^i_k(r_kx)}{\|w_k-\Pr(w_k,r_k)\|_{\tilde{L}^2(B_{r_k})}}
+\frac{r_k(a^{i,n+1}_k(r_k x)-\delta^{i,n+1})\p_{n+1}\Pr(w_k,r_k)}{\|w_k-\Pr(w_k,r_k)\|_{\tilde{L}^2(B_{r_k})}}.
\end{align*}
By \eqref{eq:r_k}, the Hölder continuity of $g_k^i$ and $a^{ij}_k$, the normalization condition (N) and by the fact that
$|\p_{n+1}\Pr(w_k,r_k)|\leq C_\epsilon r_k^{-\epsilon}$ (c.f. Lemma~\ref{lem:proj}), we deduce the bound
\begin{align}
\label{eq:bd_rhs0}
\|\hat{g}^i_k\|_{\LL(B_1)} \leq [g^{i}_k]_{C^{0,\alpha}(B_1)} r_k^{\alpha-\beta} k^{-1}
+ C_\epsilon r_k^{\alpha-\beta-\epsilon} k^{-1} [a^{ij}]_{C^{0,\alpha}(B_1)}.
\end{align}
If $\epsilon=\alpha-\beta$ in \eqref{eq:bd_rhs0}, we therefore obtain that the limit of $\hat{g}^{i}_k$ as $k\rightarrow \infty$ is zero and that
the functions $\hat g^i_k$ are uniformly bounded in $C^{0,\alpha}$.
Thus by Lemma~\ref{lem:AU} and Lemma \ref{lem:almost_linear} $\{v_k\}_{k\in\N}$ is uniformly bounded in $W^{1,2}(B_R)$ and in
$C^{0,\beta}(B_R)$ for any $\beta\in (0,1)$ and for any (arbitrary but fixed) $R>0$.
Therefore up to a subsequence, $v_k\rightarrow v_0$ locally uniformly in
$C^{0,\beta}(B_R)$ and weakly in $W^{1,2}(B_R)$ for each fixed $R>1$. Furthermore,
$v_0$ is a solution to the constant coefficient thin obstacle problem, i.e. it
satisfies the variational inequality \eqref{eq:limit_equ}.\\
To derive the latter, we observe that by a density argument it suffices
to show that for each $\eta\in C^\infty_0(B_R)$ with $v_0+\eta\geq 0$ on $B'_R$,
\begin{align}\label{eq:vari_2}
\int_{B_R} \nabla v_0\cdot \nabla \eta dx \geq 0.
\end{align}
In order to prove this, we use that $v_k$ satisfies the variational equality \eqref{eq:vari_v_k} in the following form: For all $R\in (0,1/r_k)$
\begin{equation}\label{eq:vari_v_k_a}
\begin{split}
&\int_{B_{R}}a^{ij}_k(r_k\cdot)\p_jv_k\p_i(v-v_k) dx
+\int_{B_{R}} \hat{g}^i_k \p_i(v-v_k) dx \geq 0,\\
&\text{ for all } v\in \{v\in W^{1,2}(B_{R}):v\geq 0 \text{ on } B'_{R},\ v=v_k \text{ on } \p B_{R}\},
\end{split}
\end{equation}
for which we have used that $v_k$ is a local minimizer in our convex constraint set. An argument as in the proof of Lemma \ref{lem:almost_linear} thus leads to (\ref{eq:vari_2}).\\

\emph{Step 1b: Growth.}
We begin by showing the growth estimate \eqref{eq:growth3}. To this end,
it suffices to produce an analogous growth estimate for $v_k$,
as it carries over to the limit by strong convergence. We write
\begin{align*}
v_k(Rx)&=\frac{(w_k-\Pr(w_k,r_k))(r_k Rx)}{\|w_k-\Pr(w_k,r_k)\|_{\tilde{L}^2(B_{r_k})}}\\
&=\frac{(w_k-\Pr(w_k,r_kR))(r_kRx)}{\|w_k-\Pr(w_k,r_k)\|_{\tilde{L}^2(B_{r_k})}}+\frac{(\Pr(w_k,r_kR)-\Pr(w_k,r_k))(r_kRx)}{\|w_k-\Pr(w_k,r_k)\|_{\tilde{L}^2(B_{r_k})}}.
\end{align*}
Thus,
\begin{align*}
\|v_k\|_{\tilde{L}^2(B_R)}\leq \frac{\|w_k-\Pr(w_k,r_kR)\|_{\tilde{L}^2(B_{r_kR})}}{\|w_k-\Pr(w_k,r_k)\|_{\tilde{L}^2(B_{r_k})}}
+ \frac{\|\Pr(w_k,r_kR)-\Pr(w_k,r_k)\|_{\tilde{L}^2(B_{r_kR})}}{\|w_k-\Pr(w_k,r_k)\|_{\tilde{L}^2(B_{r_k})}}.
\end{align*}
We seek to show that for each $R\geq 1$
\begin{align}
\frac{\|w_k-\Pr(w_k,r_kR)\|_{\tilde{L}^2(B_{r_kR})}}{\|w_k-\Pr(w_k,r_k)\|_{\tilde{L}^2(B_{r_k})}}\leq CR^{1+\beta},\label{eq:claim1}\\
\frac{\|\Pr(w_k,r_kR)-\Pr(w_k,r_k)\|_{\tilde{L}^2(B_{r_kR})}}{\|w_k-\Pr(w_k,r_k)\|_{\tilde{L}^2(B_{r_k})}}\leq CR^{1+\beta}. \label{eq:claim2}
\end{align}
Equations \eqref{eq:claim1} and \eqref{eq:claim2} then imply \eqref{eq:growth3}.\\

To show \eqref{eq:claim1}, for simplicity we
abbreviate $f_k(r):=\|w_k-\Pr(w_k,r)\|_{\tilde{L}^2(B_r)}$.

We recall that by the choice of $r_k$ in \eqref{eq:r_k}, we have
\begin{align*}
f_k(R r_k) \leq k (r_k R)^{1+\beta} \mbox{ and } f_k(r_k) =k r_k^{1+\beta}.
\end{align*}
As a consequence, for $R\geq 1$
\begin{align*}
\frac{f_k(Rr_k)}{f_k(r_k)}\leq \frac{k R^{1+\beta}r_k^{1+\beta}}{f_k(r_k)}=R^{1+\beta}.
\end{align*}
This proves the bound (\ref{eq:claim1}).\\
To show \eqref{eq:claim2}, we rewrite
\begin{equation}
\label{eq:proj_a}
\begin{split}
&\frac{\|\Pr(w_k,r_kR)-\Pr(w_k,r_k)\|_{\tilde{L}^2(B_{r_kR})}}{\|w_k-\Pr(w_k,r_k)\|_{\tilde{L}^2(B_{r_k})}}\\
&=\frac{f_k(Rr_k)}{f_k(r_k)}\frac{\|\Pr(w_k,r_kR)-\Pr(w_k,r_k)\|_{\tilde{L}^2(B_{r_kR})}}{f_k(r_kR)}\\
&=\frac{f_k(Rr_k)}{f_k(r_k)}\frac{R\|(\Pr(w_k,r_kR)-\Pr(w_k,r_k))(r_k\cdot)\|_{\tilde{L}^2(B_1)}}{f_k(r_k R)}\\
& \leq R^{2+\beta} \frac{\|(\Pr(w_k,r_kR)-\Pr(w_k,r_k))(r_k\cdot)\|_{\tilde{L}^2(B_1)}}{f_k(r_k R)}.
\end{split}
\end{equation}
We define (for fixed $1<R<\frac{1}{2r_k}$)
\begin{align*}
u_k(x):=\frac{(w_k-\Pr(w_k,Rr_k))(r_kx)}{f_k(r_kR)}.
\end{align*}
Then $\|u_k\|_{\tilde{L}^2(B_R)}=1$ and we note that
\begin{align*}
\frac{\|(\Pr(w_k,r_kR)-\Pr(w_k,r_k))(r_k\cdot)\|_{\tilde{L}^2(B_1)}}{f_k(r_k R)}=\|\Pr(u_k,1)\|_{\tilde{L}^2(B_1)}.
\end{align*}
To estimate $u_k$, we observe that it is a solution of the thin obstacle problem
\begin{align*}
&\int_{B_{1/r_k}}a^{ij}_k(r_k\cdot)\p_ju_k\p_i(v-u_k) dx
+\int_{B_{1/r_k}} \tilde{g}^i_k \p_i(v-u_k) dx \geq 0,\\
&\text{ for all } v\in \{W^{1,2}(B_{1/r_k}): v\geq 0\text{ on } B'_{1/r_k},\ v=u_k\text{ on } \p B_{1/r_k}\},
\end{align*}
where
\begin{align*}
\tilde{g}^i_k(x)&=\frac{r_kg^i_k(r_kx)}{f(r_k R)}
+\frac{r_k(a^{i,n+1}_k(r_k x)-\delta^{i,n+1})\p_{n+1}\Pr(w_k,r_k R)}{f(r_k R)}.
\end{align*}
By the $C^{0,\alpha}$ regularity of $a^{ij}_k, g^i_k$, assumption (N) and by the observation
that
$R^{\frac{n+1}{2}}f(r_k R)\geq f(r_k)=kr_k^{1+\beta}$ we hence obtain that
\begin{align}
\label{eq:bd_rhs}
\|\tilde{g}^i_k\|_{L^\infty(B_R)}
\leq  \frac{r_k^{\alpha-\beta}R^{\alpha + \frac{n+1}{2}}}{ k }[g^{i}_k]_{C^{0,\alpha}(B_1)}
+ \frac{r_k^{\alpha-\beta}|a_{r_k R}|R^{1+\alpha+\frac{n+1}{2}}}{k}[a^{ij}]_{C^{0,\alpha}(B_1)}.
\end{align}
Here $a_{r_kR}$ is defined by $\Pr(w_k,r_kR)=:a_{r_kR}x_{n+1}$.
By Lemma~\ref{lem:proj}, $|a_{r_k R}|\leq C_\epsilon (r_kR)^{-\epsilon}$
for $R\in (0,1/(2r_k))$.
Thus by choosing $\epsilon=\alpha-\beta$, we infer the uniform (in $k$) boundedness of $\tilde{g}^i_k$ in $B_R$ and the vanishing of the inhomogeneity $\tilde{g}^i_k$ in the limit as $k\rightarrow \infty$.
Arguing similarly one can show that $\tilde{g}^i_k$ is uniformly bounded in $C^{0,\alpha}(B_R)$.
\\
Therefore, Lemma~\ref{lem:AU} and Lemma \ref{lem:almost_linear} imply that
$\|u_k\|_{W^{1,2}(B_{R'})}+\|u_k\|_{C^{0,\beta}(B_{R'})}\leq C_{R,R'}\|u_k\|_{L^2(B_R)}$ for any $R'<R$.
Hence, up to a subsequence, $u_k\rightharpoonup u_0$ in $W^{1,2}(B_{R'})$, $u_k\rightarrow u_0$ uniformly in $B_{R'}$ and
$u_k \rightharpoonup u_0$ weakly in $L^2(B_{R})$.
Furthermore, by a similar argument as in Step 1a,
$u_0$ is a solution to the constant coefficient thin obstacle problem in $B_R$, i.e. it satisfies the variational inequality
\begin{align*}
&\int_{B_R}\nabla u_0\cdot \nabla (v-u_0) dx\geq 0,\\
&\text{ for all } v\in \{v\in W^{1,2}(B_R):v\geq 0 \text{ on } B'_R,\ v=u_0\text{ on } \p B_R\}.
\end{align*}
The uniform convergence and the normalization $u_k(0)=0$ for all $k$
imply that $u_0(0)=0$.
Using the known
linear growth of the constant coefficient solution $u_0$ at $ \Lambda_{u_0}$
(c.f. \cite{PSU})
and noting that
$\|u_0(R \cdot)\|_{\tilde{L}^2(B_1)}=\|u_0\|_{\tilde{L}^2(B_R)}\leq \liminf\limits_{k\rightarrow \infty}\|u_k\|_{\tilde{L}^2(B_R)}=1$, we obtain that
$$\|u_0\|_{\tilde{L}^2(B_1)}=\|u_0(R\cdot)\|_{\tilde{L}^2(B_{1/R})}\lesssim R^{-1}.$$
By the strong $L^2$ convergence in $B_1$, $\|u_k\|_{\tilde{L}^2(B_1)}\lesssim R^{-1}$ for $k$ sufficiently large.
Thus, $\|\Pr(u_k,1)\|_{\tilde{L}^2(B_1)}\leq \|u_k\|_{\tilde{L}^2(B_1)}\lesssim R^{-1}$.
This together with \eqref{eq:claim1} leads to
\begin{align*}
\frac{\|\Pr(w_k,r_kR)-\Pr(w_k,r_k)\|_{\tilde{L}^2(B_{r_kR})}}{\|w_k-\Pr(w_k,r_k)\|_{\tilde{L}^2(B_{r_k})}}
\lesssim R^{1+\beta}
\end{align*}
for sufficiently large values of $k$. Thus we have shown \eqref{eq:claim2}.\\

\emph{Step 2: A continuity argument.} Let $w_{k}$ and $r_k$ be as in Step 1.
We seek to prove that for any sufficiently small $\mu\in (0,1/2)$, there exist $x_k\in B'_1$ and $s_k\in (0,1)$ such that
\begin{align*}
\inf_{\gamma\in \R}\frac{\left\|(w_k-\Pr(w_k, s_kr_k, x_k))(s_kr_kx+x_k)-\gamma |x_{n+1}|\right\|_{\tilde{L}^2(B_1^+)}}{\|w_k-\Pr(w_k,s_kr_k,x_k)\|_{\tilde{L}^2(B_{s_kr_k}(x_k))}}=\mu.
\end{align*}

Indeed, since $0\in \Gamma_{w_k}$ is a free boundary point and $w_k$ is continuous, there exist points $x_k$ with
$x_k/r_k\rightarrow 0$ as $k\rightarrow \infty$ such that $w_k(x_k)>0$.
Let $\delta_k:=\dist(x_k,\Lambda_{w_k})\in (0,1)$. Then it is possible to find $\delta_k^\ell\in (0,\delta_k)$ with
$\delta_k^\ell\rightarrow 0$ as $\ell\rightarrow \infty$ such that
$w_k(r_k\delta^\ell_kx+x_k)>0$ for each $k, \ell$. We consider the blow-up limit
\begin{align*}
\tilde{u}_k(x):=\lim_{\ell\rightarrow \infty}\frac{(w_k-\Pr(w_k, r_k\delta_k^\ell, x_k))(r_k\delta_k^\ell x+ x_k)}{\|w_k-\Pr(w_k,r_k\delta_k^\ell,x_k)\|_{\tilde{L}^2(B_{r_k\delta_k^\ell}(x_k))}}.
\end{align*}
Using that $\tilde{u}_{k,\ell}:=\frac{(w_k-\Pr(w_k, r_k\delta_k^\ell, x_k))(r_k\delta_k^\ell x+ x_k)}{\|w_k-\Pr(w_k,r_k\delta_k^\ell,x_k)\|_{\tilde{L}^2(B_{r_k\delta_k^\ell}(x_k))}}$
solves the equation
\begin{align*}
\p_ia^{ij}(r_k \delta_k^{\ell}\cdot + x_k)\p_j\tilde{u}_{k,\ell}= \p_i \tilde{g}_{k,\ell}^i \text{ in } B_1,
\end{align*}
with
\begin{align*}
\tilde{g}_{k,\ell}^i(x) &:= \frac{r_k \delta_k^{\ell} g_k^i(r_k \delta_k^{\ell}\cdot + x_k)}{\|w_k-\Pr(w_k,r_k\delta_k^\ell,x_k)\|_{\tilde{L}^2(B_{r_k\delta_k^\ell}(x_k))}}
 \\
 &\quad-\frac{r_k \delta_k^{\ell} (a^{n+1,j}(r_k \delta_k^{\ell} \cdot + x_k) - \delta^{n+1,j})\p_{n+1}\Pr(w_k,r_k\delta_k^\ell,x_k)}{\|w_k-\Pr(w_k,r_k\delta_k^\ell,x_k)\|_{\tilde{L}^2(B_{r_k\delta_k^\ell}(x_k))}},
\end{align*}
we infer that $\tilde{u}_k$ solves
\begin{align*}
 \p_i a^{ij}(x_k) \p_j \tilde{u}_k = 0 \mbox{ in } B_1.
\end{align*}
Here we have used that
$\p_{\nu_+}\tilde{u}_k+\p_{\nu_-}\tilde{u}_k=0$ on $B'_1$ and
that by construction for fixed $k\in \N$
$$\|w_k-\Pr(w_k,r_k\delta_k^\ell,x_k)\|_{\tilde{L}^2(B_{r_k\delta_k^\ell}(x_k))}>0 \mbox{ uniformly in } \ell, $$
for which we recall that $x_k$ was chosen such that $w_k(x_k)>0$.
Thus, by interior regularity for solutions to elliptic equations
\begin{equation}
\label{eq:cont_1}
\begin{split}
\lim_{\ell\rightarrow \infty}\inf_{\gamma}\frac{\left\|(w_k-\Pr(w_k, r_k\delta_k^\ell, x_k))(r_k\delta_k^\ell x+ x_k)-\gamma |x_{n+1}|\right\|_{\tilde{L}^2(B_1)}}{\|w_k-\Pr(w_k, r_k\delta_k^\ell, x_k)\|_{\tilde{L}^2(B_{r_k\delta_k^\ell}(x_k))}}=1.
\end{split}
\end{equation}
By Step 1 we however know that
\begin{align*}
\lim_{k\rightarrow \infty}\inf_{\gamma}\frac{\left\|(w_k-\Pr(w_k, r_k, 0))(r_k x)-\gamma |x_{n+1}|\right\|_{\tilde{L}^2(B_1)}}{\|w_k-\Pr(w_k, r_k, 0)\|_{\tilde{L}^2(B_{r_k}(0))}}=0.
\end{align*}
Choosing $x_k \in \Omega_{w_k}$ in (\ref{eq:cont_1}) sufficiently small, e.g. such that $|x_k|^{\tilde{\alpha}}\leq \frac{1}{k} r^{1+\beta}_k$ with $\tilde{\alpha}\in(0,1)$, thus also implies that
\begin{align*}
\lim_{k\rightarrow \infty}\inf_{\gamma}\frac{\left\|(w_k-\Pr(w_k, r_k, x_k))(r_k x + x_k)-\gamma |x_{n+1}|\right\|_{\tilde{L}^2(B_1)}}{\|w_k-\Pr(w_k, r_k, x_k)\|_{\tilde{L}^2(B_{r_k}(x_k))}}=0.
\end{align*}
To observe this, we have used (\ref{eq:r_k}) for $r=r_k$ in combination with
\begin{align*}
\|\Pr(w_k, r_k, x_k)-\Pr(w_k,r_k,0)\|_{\LL(B_{r_k}(x_k))}\leq |x_k|^{\tilde{\alpha}} \leq \frac{1}{k}r_k^{1+\beta},
\end{align*}
which is a consequence of the Hölder continuity of $w_k$ (c.f. Lemma \ref{lem:almost_linear}).
Noticing that for fixed $k\in \N$ the function
\begin{align*}
[0,1] \ni s \mapsto \inf_{\gamma}\frac{\left\|(w_k-\Pr(w_k, sr_k, x_k))(sr_k x + x_k)-\gamma |x_{n+1}|\right\|_{\tilde{L}^2(B_1)}}{\|w_k-\Pr(w_k, s r_k, x_k)\|_{\tilde{L}^2(B_{s r_k}(x_k))}} \in \R,
\end{align*}
is continuous, we infer that for any $\mu\in (0,1/2)$ small, there exists $k_\mu$ sufficiently large, such that for each $k\geq k_\mu$ we can choose $s_k\in (0,1)$ and $x_k$ as claimed.  \\

\emph{Step 3: Conclusion.} Let
\begin{align*}
\tilde{v}_k(x):=\frac{(w_k-\Pr(w_k, s_kr_k,x_k))(s_kr_kx+x_k)}{\|w_k-\Pr(w_k,s_kr_k,x_k)\|_{\tilde{L}^2(B_{s_kr_k}(x_k))}}
\end{align*}
be as in Step 2. We can choose $s_k$ to be the maximal value such that for each $s\geq s_k$
\begin{align*}
\inf_{\gamma}\frac{\left\|(w_k-\Pr(w_k, s r_k,x_k))(r_k s x+ x_k)-\gamma |x_{n+1}|\right\|_{\tilde{L}^2(B_1)}}{\|w_k-\Pr(w_k, s r_k,x_k)\|_{\tilde{L}^2(B_{r_k s}(x_k))}}\leq \mu.
\end{align*}
\emph{Step 3a: Compactness.}
We claim that for an appropriate choice of $x_k \in \Omega_{w_k}$ it holds that $s_k\rightarrow 0$. Indeed, assume that this were not the case.
Then along a subsequence $s_k\rightarrow s_0>0$. We rewrite
\begin{equation}
\label{eq:rel}
\begin{split}
\tilde{v}_k(x)&
=v_k\left(s_kx+\frac{x_k}{r_k}\right)\cdot \frac{\|w_k-\Pr(w_k,r_k)\|_{\tilde{L}^2(B_{r_k})}}{\|w_k-\Pr(w_k,s_kr_k,x_k)\|_{\tilde{L}^2(B_{s_kr_k}(x_k))}}\\
& \quad - \frac{(\Pr(w_k, s_k r_k, x_k) - \Pr(w_k,r_k))(s_k r_k x + x_k)}
{\|w_k - \Pr(w_k,r_k s_k)\|_{\LL(B_{s_kr_k}(x_k))}}
.
\end{split}
\end{equation}
We consider the two summands on the right hand side of (\ref{eq:rel})
separately and first concentrate on the first term:
By Step 1, $v_k\rightarrow v_0$ strongly in $L^2(B_{3/2})$,
where $v_0(x)=\gamma|x_{n+1}|$ for some $\gamma<0$ (which is such that
$\|v_0\|_{\tilde{L}^2(B_1)}=1$).
Using this and orthogonality we deduce that
\begin{align*}
\|v_k-\Pr(v_k,s_k)\|_{L^2(B_{s_k})}&=\frac{\|w_k-\Pr(w_k,s_kr_k)\|_{\tilde{L}^2(B_{r_ks_k})}}{\|w_k-\Pr(w_k,r_k)\|_{\tilde{L}^2(B_{r_k})}}\\
&\rightarrow \|v_0\|_{\tilde{L}^2(B_{s_0})}=c_n|\gamma| s_0.
\end{align*}
Arguing as in Step 2, if the points $x_k$ are chosen such that $|x_k|$
are sufficiently small, then
\begin{align*}
\frac{\|w_k-\Pr(w_k,r_k)\|_{\tilde{L}^2(B_{r_k})}}{\|w_k-\Pr(w_k,s_kr_k,x_k)\|_{\tilde{L}^2(B_{s_kr_k}(x_k))}}\leq \frac{C}{c_n|\gamma|s_0}<\infty.
\end{align*}
In particular, along a subsequence, this term converges.
Thus, passing to the limit in the first term in (\ref{eq:rel}) yields
\begin{align*}
v_k\left(s_kx+\frac{x_k}{r_k}\right)\cdot \frac{\|w_k-\Pr(w_k,r_k)\|_{\tilde{L}^2(B_{r_k})}}{\|w_k-\Pr(w_k,s_kr_k,x_k)\|_{\tilde{L}^2(B_{s_kr_k}(x_k))}}\rightarrow C s_0|x_{n+1}| \text{ in } L^2(B_1)
\end{align*}
for some $C<\infty$. \\
We investigate the second term in (\ref{eq:rel}) and claim that for an appropriate choice of
$x_k \in \Omega_{w_k}$ it vanishes in the limit. Indeed,
\begin{align*}
&  \frac{(\Pr(w_k, s_k r_k, x_k) - \Pr(w_k,r_k))(s_k r_k x + x_k)}
{\|w_k - \Pr(w_k,r_k s_k)\|_{\LL(B_{s_kr_k}(x_k))}} \\
&= \frac{(\Pr(w_k, s_k r_k, x_k) - \Pr(w_k, r_k))(s_k r_k x + x_k)}
{\|w_k - \Pr(w_k,r_k)\|_{\LL(B_{r_k})}} \frac{\|w_k - \Pr(w_k, r_k)\|_{\LL(B_{r_k})}}
{\|w_k - \Pr(w_k, s_k r_k, x_k)\|_{\LL(B_{s_k r_k}(x_k))}}.
\end{align*}
The second factor is of the same structure as the factor from the first term, which was discussed above
and shown to be bounded.
Hence we only
study the first contribution. By a suitable (sufficiently small) choice of $x_k$ we infer that
\begin{align*}
 \frac{(\Pr(w_k, s_k r_k, x_k) - \Pr(w_k, r_k))(s_k r_k x + x_k)}
{\|w_k - \Pr(w_k,r_k)\|_{\LL(B_{r_k})}} \sim \Pr(v_k, s_k r_k)(s_k r_k x),
\end{align*}
which vanishes by strong convergence and orthogonality in the limit $k \rightarrow \infty$.
Thus, in total, we infer that
\begin{align*}
 \tilde{v}_k(x) \rightarrow C s_0 |x_{n+1}| \mbox{ in } L^2(B_1)
 \mbox{ as } k \rightarrow \infty .
\end{align*}
This however contradicts the observation that
\begin{align*}
\inf_{\gamma\in\R}\|\tilde{v}_k-\gamma|x_{n+1}|\|_{\tilde{L}^2(B_1)}=\mu\rightarrow 0,
\end{align*}
which concludes the argument for Step 3a.\\

\emph{Step 3b: Conclusion.}
With this choice of $s_k$, up to a subsequence,
$\tilde{v}_k\rightarrow \tilde{v}_0$ in $L^2(B_1)$,
where the strong $L^2(B_1)$ convergence and $\Pr(\tilde{v}_k,1)=0$ entail that $\Pr(\tilde{v}_0,1)=0$.
Indeed, due to the choice of $s_k$ and arguing similarly as in Lemma~\ref{lem:growth2}, we deduce that $\tilde{v}_k$ is a solution to the thin obstacle problem in $B_2$ with $\|\tilde{v}_k\|_{\tilde{L}^2(B_2)}$ being uniformly bounded in $k$. Moreover, by Step 2 the inhomogeneities are uniformly bounded as well. Thus by Lemma~\ref{lem:AU} we have that $\tilde{v}_k$ is uniformly bounded in $W^{1,2}(B_1)$ and $C^{0,\beta}(B_1)$. This yields the strong $L^2$ convergence along a subsequence. \\
Arguing as in Step 1a (i.e. as in the limiting argument relying on the variational inequality for $\tilde{v}_k$), we obtain that $\tilde{v}_0$ is a global solution to the thin obstacle problem for the Laplacian (in the sense of (\ref{eq:limit_equ})). We seek to show that $\tilde{v}_0(x)=a|x_{n+1}|$ for some constant $a$. Indeed, from the strong convergence and the choice of $s_k$ from above, we infer
\begin{align}\label{eq:contra2}
\inf_{\gamma\in \R}\frac{\left\|\tilde{v}_0-\gamma|x_{n+1}|\right\|_{\tilde{L}^2(B_1)}}{\|\tilde{v}_0\|_{\tilde{L}^2(B_1)}}=\mu\in (0,1/2).
\end{align}
Moreover, by the maximality of $s_k$ and by the observation that $s_k \rightarrow 0$, for each $R\geq 1$
\begin{align*}
\inf_{\gamma}\frac{\left\|\tilde{v}_0-\gamma|x_{n+1}|\right\|_{\tilde{L}^2(B_R)}}{\|\tilde{v}_0\|_{\tilde{L}^2(B_R)}}\leq \mu.
\end{align*}
By an analogue of Lemma~\ref{lem:growth2}, given any $\epsilon>0$, if $\mu=\mu_\epsilon>0$ is sufficiently small we therefore deduce that
\begin{align*}
\lim_{R\rightarrow \infty}\frac{\ln \|\tilde{v}_0\|_{\tilde{L}^2(B_R)}}{\ln R}\leq 1+\epsilon.
\end{align*}
Hence Lemma~\ref{lem:global} in combination with $\Pr(\tilde{v}_0,1)=0$ implies that
$\tilde{v}_0(x)=c+a|x_{n+1}|$ for some constants $a, c \in \R$. We claim that $c=0$. To this end, we note that $\tilde{v}_k((\hat x_k-x_k)/(s_kr_k))=0$, where $\hat x_k\in \Gamma_{w_k}$ is a free boundary point, which realizes the distance $\delta_k=\dist(x_k, \Gamma_{w_k})$. Moreover, we observe that by Step 2, $s_k$ has to satisfy $r_ks_k>\delta_k$ (i.e. $B_{r_ks_k}(x_k)\cap \Lambda_{w_k}\neq \emptyset$), which implies that $(\hat x_k-x_k)/(r_ks_k)\in B'_{1}$. Thus up to a subsequence $(\hat x_k-x_k)/(r_ks_k)\rightarrow x_0\in \overline{B'_1}$, where $\tilde{v}_0(x_0)=0$ by the uniform convergence of $\tilde{v}_k$ to $\tilde{v}_0$ in $B_1$. Thus $\tilde{v}_0(x)=a|x_{n+1}|$.
This form of $\tilde{v}_0$ however contradicts \eqref{eq:contra2}.
\end{proof}

Applying Lemma \ref{lem:growth_a} at each free boundary point and combining it with the corresponding
interior elliptic estimates, we obtain the almost optimal H\"older regularity of the solution:

\begin{prop}[Almost optimal Hölder regularity]
\label{prop:Hoelder}
Let $w$ be a solution to the variable coefficient thin obstacle problem in $B_1$ with
$a^{ij}, g^i\in C^{0,\alpha}(B_1)$ satisfying the normalization condition (N).
Then, for all $\beta \in (0,\min\{\alpha,1/2\})$ there exists a constant
$C=C(\beta, \alpha, \|a^{ij}\|_{C^{0,\alpha}}, \|g^i\|_{C^{0,\alpha}})$ such that
\begin{align*}
\|w\|_{C^{1,\beta}(B_{1/2})}\leq C \|w\|_{\LL(B_1)}.
\end{align*}
\end{prop}

\begin{proof}
By Lemma~\ref{lem:growth_a}, for any $\beta \in (0,\min\{\alpha,1/2\})$ there exists a constant
$C=C(\beta, \alpha, \|a^{ij}\|_{C^{0,\alpha}}, \|g^i\|_{C^{0,\alpha}})$ such that
 \begin{align*}
\|w-\Pr(w,r,x_0)\|_{\tilde{L}^2(B_r(x_0))} \leq C r^{1+\beta}\|w\|_{\LL(B_1)}\text{ for all }x_0 \in \Gamma_w \cap B_{1/2}'.
 \end{align*}
We set $\Pr(w,r,x_0)=:a_{x_0}(r)x_{n+1}$. By the triangle inequality and a telescope argument (see Corollary~\ref{cor:negative} for the details of a similar argument) we have for any $0<s<r<1/2$ and $x_0\in \Gamma_w\cap B_{1/2}$,
\begin{align*}
|a_{x_0}(r)-a_{x_0}(s)|\leq C r^{\beta}\|w\|_{\LL(B_1)}, \quad C=C(\beta, \alpha, \|a^{ij}\|_{C^{0,\alpha}}, \|g^i\|_{C^{0,\alpha}}).
\end{align*}
Thus the limit $\lim_{s\rightarrow 0_+} a_{x_0}(s)=:a_{x_0}$ exists, and satisfies $|a_{x_0}|\leq C$, where the constant $C$ depends on the same quantities as above. Thus we have shown that at each $x_0\in \Gamma_w\cap B_{1/2}$ there exists a linear function $\ell_{x_0}(x)=a_{x_0}x_{n+1}$ such that
$\|w-\ell_{x_0}\|_{\tilde{L}^2(B_r(x_0))}\leq Cr^{1+\beta}$. This combined with the standard elliptic estimates for the Dirichlet and Neumann problems, which are valid away from $\Gamma_w$, yields the almost optimal regularity result (see e.g. Proposition 4.22 in \cite{KRS14}).
\end{proof}

\section{Optimal Regularity}\label{sec:opt_reg}
In this section we seek to improve the almost optimal regularity result from Proposition
\ref{prop:Hoelder} to an \emph{optimal} regularity result (c.f. Theorem \ref{thm:Hoelder_opt}). To this end, we assume that
$w$ is a solution to the variable coefficient thin obstacle problem, for which the
normalization condition (N) for the $C^{0,\alpha}$ metric $a^{ij}$ and
the $C^{0,\alpha}$ inhomogeneity
$g^i$ with $\alpha\in (0,1)$ is satisfied.
Further relying on Proposition~\ref{prop:Hoelder}, we suppose that for all $\epsilon\in(0,1)$
and for all $x_0 \in \Gamma_w\cap B_{1/2}$, there exist a linear function
$\ell_{x_0}(x)=a_{x_0}x_{n+1}$ with $a_{x_0}\in \R$ and a constant
$C=C(\epsilon, \|a^{ij}\|_{C^{0,\alpha}(B_1)}, \|g^i\|_{C^{0,\alpha}(B_1)})>0$ such that
the following estimate holds true:
\begin{align*}
\|w-a_{x_0}x_{n+1}\|_{\tilde{L}^2(B_r(x_0))}\leq Cr^{1+\min\{\alpha,\frac{1}{2}\}-\epsilon}\|w\|_{\tilde{L}^2(B_1)} \text{ for all } r\in (0,1/2).
\end{align*}
In order to deduce the desired optimal regularity result, it suffices to prove
the existence of a constant $C=C(\|a^{ij}\|_{C^{0,\alpha}(B_1)}, \|g^i\|_{C^{0,\alpha}(B_1)})$ (independent of $\epsilon$) such that for all $x_0\in \Gamma_w\cap B_{1/2}$ and $r\in (0,1/2)$,
\begin{align*}
\|w-a_{x_0}x_{n+1}\|_{\tilde{L}^2(B_r(x_0))}\leq Cr^{1+\min\{\alpha,\frac{1}{2}\}}\|w\|_{\tilde{L}^2(B_1)}.
\end{align*}
As in Proposition \ref{prop:Hoelder} the combination of this with interior estimates
then implies the desired $C^{1,\min\{\alpha, 1/2\}}$ regularity of the corresponding
solutions.\\

In the following we will show this estimate for $x_0=0\in \Gamma_w$. The arguments at the other free boundary points are similar. \\

We begin by recalling the weak formulation of the equation satisfied by $\tilde{w}(x)=w(x)-a_0x_{n+1}$ in the upper and lower half balls:
It is a divergence form equation
\begin{align*}
\Delta \tilde{w}= \p_i G^i \text{ in } B_1^\pm,
\end{align*}
where
\begin{align}\label{eq:G_i}
G^i=\sum_{j}(\delta^{ij}-a^{ij})\p_j\tilde{w}+g^i+a_0 (\delta^{i,n+1}-a^{i,n+1}).
\end{align}
Using that $a^{ij}, g^i\in C^{0,\alpha}$, the assumption (N) and
Proposition~\ref{prop:Hoelder} (which gives the boundedness of $a_0$),
we in particular deduce that
\begin{equation}\label{eq:G_i2}
|G^i(x)|=|G^i(x)-G^i(0)|\leq C\left([g^i]_{{C}^{0,\alpha}}+[a^{ij}]_{{C}^{0,\alpha}}\|w\|_{\tilde{L}^2(B_1)}\right)|x|^\alpha,\quad x\in B_{1/2},
\end{equation}
where $C=C( \|a^{ij}\|_{C^{0,\alpha}}, \|g^i\|_{C^{0,\alpha}}, n)$.\\
In the full ball $\tilde{w}$ therefore solves the interior thin obstacle problem,
whose weak formulation reads
\begin{align*}
\int_{B_1}\nabla \tilde{w}\cdot \nabla \phi dx =\int_{B_1}G^i\p_i\phi dx
+ \int_{B'_1}(\p_{\nu_+}\tilde{w}+\p_{\nu_-}\tilde{w})\phi dx' \text{ for all } \phi\in C^\infty_0(B_1),
\end{align*}
where $\p_{\nu_+}+\p_{\nu_-}$ is defined as in \eqref{eq:normal}.
We note that by \eqref{eq:normal} and our definition of $\tilde w$,
$\p_{\nu^+}\tilde{w}(x)+\p_{\nu_-}\tilde{w}(x)=\p_{\nu_+}w(x)+\p_{\nu_-}w(x)$
for $x\in B'_1$. Furthermore, by the up to $B'_{1/2}$ regularity
from Proposition~\ref{prop:Hoelder},
$\p_{\nu_+}w(x)+\p_{\nu_-}w(x)\in C^{0,\min\{1/2,\alpha\}-\epsilon}(B^\pm_{1/2})$.
Thus, the complementary boundary conditions are satisfied in a pointwise sense and they turn into
\begin{align*}
\p_{\nu_+}w(x)+\p_{\nu_-}w(x)&
\left\{
\begin{array}{ll}
&=0 \text{ if } x\in B'_1\cap \{w>0\},\\
&\geq 0 \text{ if } x\in B'_1\cap \{w=0\}.
\end{array}
\right.
\end{align*}
Therefore, we can also interpret $\Delta \tilde{w} \in H^{-1}$ as
\begin{align}\label{eq:main2}
\Delta \tilde{w}= \p_iG^i-(\p_{\nu_+}w+\p_{\nu_-}w)\mathcal{H}^{n}\lfloor (B'_1\cap \Lambda_w) \text{ in } B_1.
\end{align}
For notational simplicity we will assume that $a_0=0$ and thus $\tilde{w}=w$. Furthermore, we assume that $\|w\|_{\tilde{L}^2(B_1)}=1$.\\

The remainder of this section is organized as follows: In Section \ref{subsec:epi} we introduce the Weiss energy, which serves as a central tool in deriving compactness in our optimal regularity proof. Here we exploit the epiperimetric inequality in the form of \cite{FS16}, in order to derive an explicit decay rate for the Weiss energy. Using the full strength of the estimate from \cite{FS16}, it is possible to implement this strategy even for the situation of only Hölder continuous metrics by perturbative arguments (c.f. Corollary \ref{cor:perturb_epi}). In Section \ref{subsec:opt} we heavily use the decay of the Weiss energy to carry out our compactness argument (c.f. Lemma \ref{lem:negative}). This ultimately leads to the central growth estimates and hence to the optimal regularity result.

\subsection{The Weiss energy and the epiperimetric inequality}\label{subsec:epi}
In the sequel, we consider the Weiss energy
\begin{align*}
W(r,w)=W_{3/2}(r,w):=\frac{1}{r^{n+2}}\int_{B_r}|\nabla w|^2 dx
-\frac{3}{2}\frac{1}{r^{n+3}}\int_{\p B_r} w^2 \dH,\quad r\in (0,1).
\end{align*}
It is easy to verify that
\begin{equation}\label{eq:scaling_Weiss}
W(r,w)=W(1,w_r),\quad w_r(x):=\frac{w(rx)}{r^{3/2}},  \ r\in (0,1).
\end{equation}
It is known (c.f. Theorem 9.24 in \cite{PSU}) that if $u$ is a solution to the thin
obstacle problem with $a^{ij}=\delta^{ij}$ and $g^i=0$, then for any real number
$\kappa\geq 0$,
$$r\mapsto W_\kappa(r,u):=\frac{1}{r^{n-1+2\kappa}}\int_{B_r}|\nabla u|^2 dx
-\frac{\kappa}{r^{n+2\kappa}}\int_{\p B_r} u^2 \dH $$ is nondecreasing
on $(0,1)$.  Moreover,  $W_\kappa(\cdot,u)$ is constant, iff $u$ is a homogeneous solution of
degree $\kappa$.
In the \emph{variable} coefficient case we do not try to mimic this monotonicity,
but will instead use the Weiss energy as a
tool that yields compactness and controls the ``vanishing order'' of our solution.\\

We introduce the space $\mathcal{E}$, which is defined as the set of the three
halves blow-up solutions
(modulo rotation and scaling), which in combination with the resulting projection operator
will play a central role in this section:

\begin{defi}
\label{defi:proj_3/2}
Let
$$\mathcal{E}:=\{c \Ree((x'\cdot \xi)+i |x_{n+1}|)^{3/2}: c\geq 0,\ \xi\in \R^n,\ |\xi|=1\}.$$
For $w\in L^2(\p B_r)$ we consider the $L^2$ projection of $w$ onto $\mathcal{E}$,
i.e. we define $\overline{\Pr}(w,r)\in \mathcal{E}$ such that
$$\int_{\p B_r(0)}|w-\overline{\Pr}(w,r)|^2 \dH
=\min_{p\in \mathcal{E}}\int_{\p B_r(0)}|w-p|^2 \dH .$$
More generally, we set
\begin{align*}
\overline{\Pr}(w,r,x_0):= \argmin\limits_{p\in \mathcal{E}}\left\{ \int_{\p B_r(0)}|w(\cdot + x_0)-p(\cdot)|^2 \dH \right\}.
\end{align*}
We further use the notation
\begin{align*}
 h_{3/2}(x):= \Ree(x_n + i |x_{n+1}|)^{3/2}
\end{align*}
to denote the model solution.
\end{defi}

\begin{rmk}
\label{rmk:proj}
We emphasize that in contrast to the analogous projection onto $\mathcal{P}$ from Definition \ref{defi:proj}, in Definition \ref{defi:proj_3/2} we now define our projection with respect to the $L^{2}(\p B_r)$ instead of the $L^2(B_r)$ topology. This is due to our compactness argument, which relies on the Weiss energy. As the Weiss energy contains a \emph{boundary normalized term}, boundary projections are better suited in this framework. In Corollary \ref{cor:negative} we will however pass from information on the boundary of $B_r$ to information on the full solid ball.
\end{rmk}

With the Weiss energy and the space $\mathcal{E}$ at hand, we recall the statement of the epiperimetric inequality, which
has been derived in \cite{GPSVG15}, \cite{FS16}. We heavily
rely on it (in the formulation of \cite{FS16}) in deducing quantitative control over the Weiss energy
(c.f. Lemma \ref{cor:perturb_epi}).

\begin{thm}[Theorem 3.1 in \cite{FS16}]
\label{thm:epi}
There exists $\kappa\in (0,1)$ with the property that for any $c\in W^{1,2}(\partial B_1)$
with $c\geq 0$ on $\p B'_1$,
there exists $u\in W^{1,2}(B_1)$ with $u=c$ on $\p B_1$ and $u\geq 0$ on $B'_1$ such that
\begin{align*}
W(1,u)\leq  (1-\kappa)W(1,\tilde{c}).
\end{align*}
Here $\tilde{c}(x)=|x|^{3/2}c(x/|x|)$ is the $3/2$-homogeneous extension of $c$.
\end{thm}

Note that formally
 Theorem 3.1 in \cite{FS16} has the additional assumption that $\inf_{h\in \mathcal{E}}\|c-h\|_{W^{1,2}(\p B_1)}\leq \delta$ for some small $\delta>0$. However this assumption can be removed by a simple scaling argument. Indeed, the epiperimetric inequality (with constant $\kappa$) holds for $c$, iff it holds for $\lambda c$ for $\lambda>0$, since the epiperimetric inequality is homogeneous. \\
In the sequel, we will use the epiperimetric inequality as a central tool, which entails a geometric decay of the Weiss energy associated with a solution to the thin obstacle problem (c.f. Corollary~\ref{cor:perturb_epi}). In order to obtain this, we begin by proving
the following auxiliary property (which is of algebraic nature).

\begin{lem}
\label{lem:deri_energy}
Given any $u\in C^1(\overline{B_1})$, let $u_r(x):=u(rx)/r^{3/2}$ and let $\tilde{u}_r(x):=|x|^{3/2} u_r(x/|x|)$ be the $3/2$-homogeneous extension of $u_r\big|_{\p B_1}$. Then,
\begin{align*}
\frac{d}{dr}W(r,u)\geq \frac{n+2}{r}\left(W(1,\tilde{u}_r)-W(r,u)\right),\quad r\in (0,1).
\end{align*}
\end{lem}

\begin{proof}
The proof follows from a direct computation:
First we note that if $\tilde{c}(x)=|x|^{3/2}c(x/|x|)$ for some $c:\p B_1\rightarrow \R$,
then it is possible to compute that
\begin{align}
\label{eq:homo_energy}
W(r,\tilde{c})=\frac{1}{n+2}\left(\int_{\p B_1}|\nabla_\theta c|^2 \dH
-\frac{6n+3}{4}\int_{\p B_1}|c|^2 \dH \right),
\end{align}
which is independent of $r$.
Next given any $u\in C^1(\overline{B_1})$, we compute $\frac{d}{dr}W(r,u)$:
\begin{align*}
\frac{d}{dr}W(r,u)&=\frac{-(n+2)}{r^{n+3}}\int_{B_r}|\nabla u|^2dx
+\frac{1}{r^{n+2}}\int_{\p B_r}|\nabla u|^2 \dH\\
&+\frac{3}{2}\frac{(n+3)}{r^{n+4}}\int_{\p B_r}u^2 \dH
-\frac{3}{2}\frac{1}{r^{n+3}}\left(\frac{n}{r}\int_{\p B_r}u^2 \dH
+ 2\int_{\p B_r}u\p_\nu u \dH \right)\\
&=-\frac{n+2}{r}W(r,u)-\frac{3}{2}\frac{n-1}{r^{n+4}}\int_{\p B_r}u^2 \dH\\
&\quad
+\frac{1}{r^{n+2}}\int_{\p B_r}|\nabla u|^2 \dH -\frac{3}{r^{n+3}}\int_{\p B_r}u\p_\nu u \dH.
\end{align*}
Applying Cauchy-Schwartz to the last term and setting $\tilde{u}_r(x)=
|x|^{3/2}u_r(\frac{x}{|x|})$, we obtain
\begin{align*}
\frac{d}{dr}W(r,u)&\geq -\frac{n+2}{r}W(r,u)
+\frac{1}{r^{n+4}}\int_{\p B_r}|\nabla _\theta u|^2 \dH\\
&\quad
-\left(\frac{3(n-1)}{2}+\frac{9}{4}\right)\frac{1}{r^{n+4}}\int_{\p B_r}u^2 \dH\\
&=-\frac{n+2}{r}W(r,u)+\frac{1}{r}\int_{\p B_1}|\nabla _\theta u_r|^2 \dH
-\frac{6n+3}{4r}\int_{\p B_1}u_r^2 \dH\\
&\stackrel{(\ref{eq:homo_energy})}{=}\frac{n+2}{r}\left(W(1,\tilde{u}_r)-W(r,u)\right).
\end{align*}
This shows the desired estimate.
\end{proof}

With this at hand, we invoke the epiperimetric inequality to
deduce a differential inequality for the Weiss energy. This allows us to
derive an explicit decay rate for the Weiss energy.

\begin{cor}\label{cor:perturb_epi}
Assume that $w$ is a local minimizer to the functional
\begin{align*}
J(u)=\int_{B_1}a^{ij}\p_iu\p_ju + g^i\p_i u dx,
\quad u\in \mathcal{K}=\{v\in W^{1,2}(B_1): v\geq 0 \text{ on } B'_1\},
\end{align*}
where $a^{ij}\in C^{0,\alpha}(B_1), g^i\in C^{0,\alpha}(B_1)$ with
$\alpha\in (1/2,1)$ and where the normalization assumption (N) is satisfied.
Then there exists a constant $\delta_0>0$ such that if
$[a^{ij}]_{{C}^{0,\alpha}(B_1)}+[g^i]_{{C}^{0,\alpha}(B_1)}\leq \delta_0$,
then
\begin{align*}
W(r,w)&\leq \left(\frac{r}{r_1}\right)^{\beta}W(r_1,w)+C\delta_0r^\beta,
\quad r\in (0,r_1]
\end{align*}
for some $\beta\in (0,\frac{1}{2}(\alpha-\frac{1}{2}))$,
$C=C(n,\alpha,\delta_0,
\|w\|_{\LL(B_1)})>0$ and $r_1\in (0,1/2]$ arbitrary.
\end{cor}

\begin{proof}
The functional $J$ can be viewed as a perturbation of the classical Dirichlet energy.
More precisely, setting
\begin{align*}
 J(u,r):=\int_{B_r}a^{ij}\p_iu\p_ju + g^i\p_i u dx,
\end{align*}
the normalization condition (N) in combination with the
$C^{0,1/2-\epsilon}_{loc}(B_1)$
regularity of $|\nabla u|$ implies that
\begin{align*}
&\left(1-[a^{ij}]_{{C}^{0,\alpha}}r^{\alpha}\right)\int_{B_r}|\nabla u|^2 dx
- C[g^i]_{{C}^{0,\alpha}}r^{\alpha+n+3/2-\epsilon}
\leq J(u,r)\\
& \leq \left(1+[a^{ij}]_{{C}^{0,\alpha}}r^{\alpha}\right)\int_{B_r}|\nabla u|^2 dx
+ C[g^i]_{{C}^{0,\alpha}}r^{\alpha+n+3/2-\epsilon},
\end{align*}
where $C=C(n,\alpha,\|a^{ij}\|_{C^{0,\alpha}},\|g^i\|_{C^{0,\alpha}}, \epsilon)$ and $\epsilon
\in (0,1/2)$.
Furthermore, recalling the definition of the Weiss energy and denoting
$e(r):=[a^{ij}]_{{C}^{0,\alpha}}r^{\alpha}$ for simplicity, we have
\begin{align*}
&\left(1-e(r)\right) W(r,u)
-C[g^i]_{{C}^{0,\alpha}} r^{\alpha-1/2-\epsilon}
+ \left(1-e(r)\right)\frac{3}{2r^{n+3}} \int_{\p B_r}u^2 \dH\\
&\leq \frac{J(u,r)}{r^{n+2}}\\
&\leq \left(1+e(r)\right) W(r,u)+C[g^i]_{{C}^{0,\alpha}} r^{\alpha-1/2-\epsilon}
+\left(1+e(r)\right)\frac{3}{2r^{n+3}} \int_{\p B_r}u^2 \dH.
\end{align*}
Thus, the epiperimetric inequality (c.f. Theorem~\ref{thm:epi}) and a two fold comparison argument
(first applied to solutions
of the standard Dirichlet energy, then to the variable coefficient energy $J(u,r)$, where we use that
$w$ is a local minimizer) yield:
\begin{align*}
W(r,w) &=W(1,w_r)
\leq (1-\kappa)\frac{1+ e(r)}{1-e(r)}W(1,\tilde{w}_r)
\\
&+\frac{2C}{1-e(r)}[g^i]_{{C}^{0,\alpha}} r^{\alpha-1/2-\epsilon}
+ \frac{3 e(r)}{(1-e(r))r^{n+3}} \int\limits_{\partial B_r} w^2 \dH.
\end{align*}
Here $w_r(x):= r^{-3/2} w(r x)$ and
$\tilde{w}_r(x)=|x|^{3/2}w_r(x/|x|)$.
We emphasize that at this point it was crucial to use the epiperimetric inequality in the form
of \cite{FS16}, c.f. Theorem \ref{thm:epi}, as the assumptions of the analogous epiperimetric
inequality from \cite{GPSVG15} are more restrictive and would already require stronger control
on the $3/2$-homogeneous rescalings $\tilde{w}_r(x)$.\\
Thus, if $[a^{ij}]_{{C}^{0,\alpha}}$ is sufficiently small,
we have $\tilde{\kappa}:=-\frac{2e(r)}{1-e(r)}+\kappa \frac{1+e(r)}{1-e(r)}\in (0,1)$. This together with the $C^{1,1/2-\epsilon}_{loc}$ regularity of $w$ gives
\begin{align*}
W(r,w)\leq (1-\tilde{\kappa})W(1,\tilde{w}_r)+C\delta_0r^{\alpha-1/2-\epsilon},
\end{align*}
where here and in the sequel the constant $C>0$ differs from line by line by a universal constant.
Applying Lemma~\ref{lem:deri_energy} hence results in
\begin{align*}
\frac{d}{dr}W(r,w)&\geq \frac{n+2}{r}\left(W(1,\tilde{w}_r)-W(r,w)\right)\\
&\geq \frac{n+2}{r} \frac{\tilde{\kappa}}{1-\tilde{\kappa}} W(r,w)
-\frac{n+2}{(1-\tilde{\kappa})r}C\delta_0r^{\alpha-1/2-\epsilon}\\
&\geq \frac{n+2}{r} \frac{\tilde{\kappa}}{1-\tilde{\kappa}} W(r,w)
-C\delta_0 r^{\alpha-3/2-\epsilon}.
\end{align*}
Using an integrating factor, this turns into
\begin{align*}
 \frac{d}{dr}\left( W(r,w) r^{-\beta} \right)
 \geq -C\delta_0 r^{\alpha-3/2-\beta-\epsilon},\quad \beta= \frac{(n+2)\tilde{\kappa}}{1-\tilde{\kappa}}.
\end{align*}
We note that by choosing $[a^{ij}]_{C^{0,\alpha}}$ sufficiently small, we may without loss of generality assume that $0<\beta<\alpha-1/2-\epsilon$.
An integration thus yields for $r<r_1$,
\begin{equation}\label{eq:Weiss_decay}
 W(r,w) \leq \left(\frac{r}{r_1}\right)^{\beta}W(r_1,w)
 +C\delta_0 r^\beta\left(r_1-r\right)^{\alpha-1/2-\beta-\epsilon}.
\end{equation}
Choosing $\beta +\epsilon\in (0,\frac{1}{2}(\alpha-\frac{1}{2}))$ sufficiently small, we obtain the desired estimate.
\end{proof}

\subsection{Optimal regularity}
\label{subsec:opt}
As the central result of this subsection, we derive the optimal regularity of solutions to the variable coefficient thin obstacle problem. Here the crucial ingredient is a quantitative asymptotic expansion of solutions at the free boundary in terms of limiting profiles in $\mathcal{E}$ (c.f. Lemma \ref{lem:negative}). This is obtained by exploiting the decay rate of the Weiss energy.\\

We begin our discussion of the optimal regularity properties with a first approximation result, which also uses the decay of the Weiss energy as an essential ingredient in obtaining compactness. This approximation result roughly states that there is a gap of the decay rate of the solution around free boundary points, i.e. solutions either decay faster than $r^{3/2+\epsilon_0}$ or stay close to $c_nh_{3/2}$ (after an $L^2$-normalization).

\begin{lem}
\label{lem:blow_up}
Let $w:B_1 \rightarrow \R$ be a solution to the thin obstacle problem
satisfying the condition (N). Assume that $\|w\|_{\LL(B_1)}=1$ and that $0 \in \Gamma_w$. Let $\delta_0,\beta$ be the constants from Corollary~\ref{cor:perturb_epi}. Assume that $[a^{ij}]_{{C}^{0,\alpha}}+[g^i]_{{C}^{0,\alpha}}\leq \delta_0$.
Then, for any $\mu>0$ and
$\epsilon_0\in (0,\frac{\beta}{4})$ fixed,
there exists
$r_0=r_0(\mu, \|a^{ij}\|_{C^{0,\alpha}}, \|g^i\|_{C^{0,\alpha}},\epsilon_0)>0$
such that for those $r\in (0,r_0)$ satisfying
$\|w\|_{\tilde{L}^2(\p B_r)}\geq r^{\frac{3}{2}+\epsilon_0}$, it holds
\begin{align*}
\|w-\overline{\Pr}(w,r)\|_{\tilde{L}^2(\p B_r)}\leq \mu \|w\|_{\tilde{L}^2(\p B_r)}.
\end{align*}
In particular,
\begin{align*}
\|\overline{\Pr}(w,r)\|_{\LL(\partial B_r)} \geq (1-\mu)\|w\|_{\LL(\p B_r)}.
\end{align*}
\end{lem}

\begin{proof}
The proof of the result follows from a compactness argument.
Assume that the statement were wrong.
Then there existed $\mu_0>0$, a sequence of solutions $w^k$ and a sequence of
radii $r_k\rightarrow 0_+$ such that $\|w^k\|_{\LL(B_1)}=1$, $0\in \Gamma_{w^k}$ and
$\|w^k\|_{\tilde{L}^2(\p B_{r_k})}\geq r^{\frac{3}{2}+\epsilon_0}_k$,
but with $\|w^k-\overline{\Pr}(w^k,r_k)\|_{\tilde{L}^2(\p B_{r_k})}>\mu_0 \|w^k\|_{\tilde{L}^2(\p B_{r_k})}$.\\
Let $w^k_{r_k}(x):=w^k(r_k x)/\|w^k\|_{\tilde{L}^2(\p B_{r_k})}$.
Invoking Corollary \ref{cor:perturb_epi} with $r=r_k$ and $r_1 = \frac{1}{2}$, we obtain
\begin{align*}
W(1,w^k_{r_k})=\frac{r_k^3}{\|w^k\|_{\LL(\p B_{r_k})}^2}W(r_k,w^k)\leq C\frac{r_k^{3+\beta}}{\|w^k\|_{\LL(\p B_{r_k})}^2}.
\end{align*}
We emphasize that we in particular made use of the bound
$\|w^k\|_{\LL(B_1)}= 1$,
in order to estimate the contribution of $W(1/2,w^k)$ in an in $k$ uniform manner. Using the lower bound $\|w^k\|_{\LL(\p B_{r_k})} \geq r^{3/2+\epsilon_0}_k$ and letting $\epsilon_0<\beta/4$, we obtain
\begin{equation}\label{eq:bound_W}
W(1,w^k_{r_k})\leq C r_k^{\beta/2}.
\end{equation}
Recalling that $W(1,w^k_{r_k})=\|\nabla w_{r_k}^k\|_{L^2(B_1)}^2-\frac{3}{2}\|w_{r_k}^k\|_{L^2(\p B_1)}^2$ and that $\|w^k_{r_k}\|_{\LL(\p B_1)}=1$, we infer that
\begin{equation}
 \label{eq:upp}
 \begin{split}
  \|\nabla w_{r_k}^k\|_{\LL(B_1)}^2
&\leq C r^{\beta/2}_k + \frac{3}{2}.
 \end{split}
\end{equation}
Thus, up to a subsequence $w^k_{r_k}\rightarrow w_0$ weakly in $W^{1,2}(B_1)$ and
strongly in both $L^2(B_1)$ and $L^2(\p B_1)$ for some $w_0\in W^{1,2}(B_1)$ with
$\|w_0\|_{\tilde{L}^2(\p B_1)}= 1$.
Since $w^k_{r_k}$ solves
\begin{align*}
\Delta w^k_{r_k}=\frac{r_k}{\|w^k\|_{\tilde{L}^2(\p B_{r_k})}}\p_iG^i_k(r_k\cdot) &\text{ in } B_{1/r_k}^\pm,\\
w^k_{r_k}\geq 0, \quad \p_{\nu_+}w^k_{r_k}+\p_{\nu_-}w^k_{r_k}\geq 0, \quad w^k_{r_k}(\p_{\nu_+}w^k_{r_k}+\p_{\nu_-}w^k_{r_k})=0 &\text{ on } B'_{1/r_k},
\end{align*}
and since $\|w^k\|_{\tilde{L}^2(\p B_{r_k})}\geq r^{\frac{3}{2}+\epsilon_0}$ with $\epsilon_0\in (0,\beta/4)$
and $|G^i_k(x)|\leq C|x|^\alpha$ and $\alpha\in (1/2,1)$ by \eqref{eq:G_i},
the inhomogeneities $\frac{r_k}{\|w^k\|_{\tilde{L}^2(\p B_{r_k})}}G^i_k(r_k\cdot)$
are uniformly bounded in $C^{0,\alpha}(B_1)$ and vanish in the limit as
$k\rightarrow \infty$.
By the interior H\"older regularity (c.f. Proposition~\ref{prop:Hoelder}) we furthermore obtain that
$w^k_{r_k}\in C^{1,1/2-\epsilon}_{loc}(B_{1}^\pm)$. Hence,
up to a subsequence $w^k_{r_k}\rightarrow w_0$ in $C^{1,1/2-\epsilon}_{loc}(B_1^\pm)$, where $w_0$ is
a solution to the homogeneous, constant coefficient thin obstacle problem for the Laplacian
in $B_1$. Moreover, the the almost optimal regularity result combined with the assumption that $0\in \Gamma_{w_k}$ yields that $0\in \Gamma_{w_0}$. \\
We seek to show that $w_0\in \mathcal{E}\setminus \{0\}$.
Indeed, on the one hand, by weak lower semi-continuity,
the convergence $r_k \rightarrow 0$ and by (\ref{eq:upp})
we infer
\begin{align*}
\int_{B_1}|\nabla w_0|^2 dx
\leq \liminf _{k\rightarrow\infty}\int_{B_1}|\nabla w^k_{r_k}|^2 dx
\leq \frac{3}{2}.
\end{align*}
On the other hand, the lower bound for the frequency function (associated with the zero thin obstacle problem for the Laplacian at $0\in \Gamma_{w_0}$) yields
\begin{align*}
\frac{3}{2}\leq N(1,w_0)=\int_{B_1}|\nabla w_0|^2 dx.
\end{align*}
Hence, $N(1,w_0)=\frac{3}{2}$ and therefore the monotonicity of the frequency function
(for the thin obstacle problem with $a^{ij}=\delta^{ij}$)
implies
\begin{align*}
 N(r,w_0) = \frac{3}{2} \mbox{ for all } r\in(0,1).
\end{align*}
As a consequence, $w_0$ is homogeneous of degree $3/2$. By the
classification of homogeneous solutions, this results in the identity
$w_0(x)=c_n \Ree(x_n+i|x_{n+1}|)^{3/2}$ up to a rotation, i.e. $w_0\in \mathcal{E}\setminus\{0\}$.
This however is a contradiction to the strong $L^2$ convergence of $w^k_{r_k}$ to $w_0$
and to the assumption that
$\|w^k_{r_k}-\overline{\Pr}(w^k_{r_k},1)\|_{\LL(\p B_1)}>\mu_0$ for all $k$.
\end{proof}

Heading towards a convergence rate for the decay of a solution of the thin obstacle problem towards $\mathcal{E}$, we next establish the following lemma, which controls the Weiss energy of the difference of a solution and an arbitrary (but fixed) element of $\mathcal{E}$ in terms of the Weiss energy of the given solution.

\begin{lem}
\label{lem:orth_comp}
Let $w\in W^{1,2}(B_1)$ with $w\geq 0$ on $B'_1$. Then for any function $p_{3/2} \in \mathcal{E}$, we have
\begin{align}
\label{eq:proj}
W(r,w-p_{3/2}) \leq W(r,w), \quad r\in (0,1).
\end{align}
\end{lem}

\begin{proof}
The proof follows from a direct computation:
\begin{align*}
W(r,w-p_{3/2})&=W(r,w)+W(r,p_{3/2})\\
&-2\left(\frac{1}{r^{n+2}}\int_{B_r}\nabla w\cdot \nabla p_{3/2} dx
-\frac{3}{2r^{n+3}}\int_{\p B_r} w p_{3/2} \dH \right).
\end{align*}
By the Euler identity,
$\p_\nu p_{3/2}= \frac{x}{|x|}\cdot \nabla p_{3/2}=\frac{3}{2r}p_{3/2}$ on $\p B_r$,
and the facts that $\Delta p_{3/2}\leq 0$ in $B_1$,
$\supp (\Delta p_{3/2})\subset B'_1$ and $w\geq 0$ on $B'_1$, we have
\begin{align*}
\frac{1}{r^{n+2}}\int_{B_r}\nabla w\cdot \nabla p_{3/2}dx
-\frac{3}{2r^{n+3}}\int_{\p B_r} w p_{3/2}\dH =-\frac{1}{r^{n+2}}\int_{B_r} w\Delta p_{3/2} dx\geq 0.
\end{align*}
This combined with $ W(r,p_{3/2})\equiv 0$ yields \eqref{eq:proj}.
\end{proof}

Having established these facts on the Weiss energy, as the key point of this section we address the convergence rate of solutions to the variable coefficient thin obstacle problem to the space $\mathcal{E}$:

\begin{lem}\label{lem:negative}
Under the assumptions of Lemma~\ref{lem:blow_up}, for all constants
$\epsilon_0 \in (0,\beta/8)$
and for all $r\in (0,1/2)$ it holds
\begin{align*}
\|w-\overline{\Pr}(w,r)\|_{\LL(\p B_r)}
\leq C_0  r^{\epsilon_0}\max\{\|w\|_{\tilde{L}^2(\p B_r)}, r^{3/2}\},
\end{align*}
where
$C_0=C_0(n,\| a^{ij}\|_{C^{0,\alpha}(B_1)},\| g^{i}\|_{C^{0,\alpha}(B_1)})>0$.
\end{lem}

\begin{proof}
\emph{Step 1:}
We first observe that if $\|w\|_{\tilde{L}^2(\p B_r)}\leq r^{3/2+\epsilon_0}$,
the estimate follows directly from
$\|w-\overline{\Pr}(w,r)\|_{\tilde{L}^2(\p B_r)}\leq \|w\|_{\tilde{L}^2(\p B_r)}
\leq r^{3/2+\epsilon_0}$. \\
Thus, we will only focus on the case, in which
$\|w\|_{\tilde{L}^2(\p B_r)}> r^{3/2+\epsilon_0}$.
\\

We argue by contradiction and suppose the estimate were not true. Then there
existed a value of $\epsilon_0 \in (0,\beta/8)$, a
sequence of radii $r_k$ and a sequence of solutions $w^k$ with $\|w^k\|_{\LL(B_1)}=1$, $0\in \Gamma_{w^k}$,
and $\|w^k\|_{\tilde{L}^2(\p B_{r_k})}>  r_k^{3/2+\epsilon_0}$,
such that
\begin{align*}
\|w^k-\overline{\Pr}(w^k,r_k)\|_{\tilde{L}^2(\p B_{r_k})}\geq k  r_k^{\epsilon_0}h_k(r_k),
\end{align*}
where $h_k(r):= \max\{\|w^{k}\|_{\tilde{L}^2(\p B_{r})}, r^{3/2}\}$.
By Lemma~\ref{lem:blow_up} and the fact that $h_k(r_k)\geq  \|w^k\|_{\tilde{L}^2(\p B_{r_k})}$ necessarily $r_k\rightarrow 0_+$.
Moreover, we can choose $r_k \rightarrow 0$ and $C_k\rightarrow \infty$ such that
\begin{equation}\label{eq:negative_contra3}
\begin{split}
&\|w^k-\overline{\Pr}(w^k,r)\|_{\tilde{L}^2(\p B_r)}\leq C_k  r_k^{\epsilon_0}h_k(r)
\quad \mbox{ for all } r\in (0,r_k]\cap I_k , \\
&\|w^k-\overline{\Pr}(w^k,r)\|_{\tilde{L}^2(\p B_r)}\leq C_k  r^{\epsilon_0}h_k(r)
\quad \mbox{ for all } r\in  [r_k,1)\cap I_k,\\
&\|w^k-\overline{\Pr}(w^k,r_k)\|_{\tilde{L}^2(\p B_{r_k})}= C_k r^{\epsilon_0}_k h_k(r_k),\\
&\text{ where } I_k:=\{r\in (0,1]:\|w^k\|_{\tilde{L}^2(\p B_r)}\geq r^{3/2+\epsilon_0}\} .
\end{split}
\end{equation}
(For example, given $k$ assuming $I_k$ is nonempty, we let $\widetilde{r_k}$ be the maximal $r\in I_k$ such that $f_k(r):=\|w^k-\overline{\Pr}(w^k,r)\|_{\tilde{L}^2(\p B_r)}/h_k(r)\geq kr^{\epsilon_0}$. Such a sequence of radii $\widetilde{r_k}$ exists by the contradiction assumption. Then we define $r_k\in [0,\widetilde{r_k}]$ such that $f_k(r_k)=\max_{r\in [0,\widetilde{r_k}]\cap I_k} f_k(r)$. Note that $r_k\neq 0$, since $\lim_{r\in I_k, r\rightarrow 0_+}f_k(r)=0$ by Lemma~\ref{lem:blow_up} and the definition of $h_k$. Let $C_k:=f_k(r_k)/r_k^{\epsilon_0}$. It is not hard to verify that \eqref{eq:negative_contra3} holds with such a choice of $r_k$ and $C_k$.)
Let
$$v_k(x):=\frac{w^k_{r_k}(x)-\overline{\Pr}(w^k_{r_k},1)(x)}
{\|w^k_{r_k}-\overline{\Pr}(w^k_{r_k},1)\|_{\LL(\p B_1)}},\quad w^k_{r_k}(x):=\frac{w^k(r_kx)}{\|w^k\|_{\tilde{L}^2(\p B_{r_k})}}.$$
In particular, $\|v_k\|_{\tilde{L}^2(\p B_1)}=1$. We invoke \eqref{eq:bound_W} and Lemma~\ref{lem:orth_comp} in combination with $\|w\|_{L^2(B_1)}=1$
to observe that
\begin{align*}
 W(1,v_k) \leq \frac{W(1,w_{r_k}^k)}{ \|w^k_{r_k} - \overline{\Pr}(w_{r_k}^k,1)\|_{\LL(\p B_1)}^2}
 \leq  \frac{Cr^{\beta/2}_k}{\|w^k_{r_k} - \overline{\Pr}(w_{r_k}^k,1)\|_{\LL(\p B_1)}^2}.
 \end{align*}
Equation (\ref{eq:negative_contra3}) and the choice of $\epsilon_0\in (0,\beta/8)$ thus yield
\begin{align*}
 W(1,v_k)  \leq C r^{\beta/4}_k.
 \end{align*}
Consequently,
\begin{align*}
\|\nabla v_k\|_{L^2(B_1)}^2
&\leq W(1,v_k) + \frac{3}{2}\|v_k\|_{L^2(\p B_1)}^2
\leq C_{\beta} r^{\beta/4}_k + \frac{3}{2}.
 \end{align*}
Then up to a subsequence $v_k\rightarrow v_0$ weakly in $W^{1,2}(B_1)$ and strongly both
in $L^2(B_1)$ and $L^2(\p B_1)$.
Moreover,
the limit $v_0$ solves (up to a rotation)
\begin{equation}
\label{eq:linear}
 \begin{split}
\Delta v_0&=0 \text{ in } B_1^+\cup B_1^-,\\
\p_{n+1}^+v_0 + \p_{n+1}^-v_0&=0 \text{ on } B_1' \cap \{x_n>0\},\\
v_0&=0 \text{ on } B_1'\cap \{x_n< 0\}.
\end{split}
\end{equation}
Indeed, by Lemma~\ref{lem:blow_up} $w^k_{r_k}\rightarrow c_n p^{\nu}_{3/2}$ in $C^{1,\beta}_{loc}(B_1^\pm)$
and after a rotation we may assume $\nu=e_n$. Thus for each $t>0$ there exists
$k_t\in \N$, such that for $k\geq k_t$ the functions $v_k(Q_kx', x_{n+1})$, where $Q_k\in SO(n)$ (which is chosen to be the rotation associated with $\Pr(w_{r_k}^k,1)$), satisfy
\begin{equation}\label{eq:linear_2}
\begin{split}
\Delta v_k&= \frac{r_k\p_i G^i_k(r_k\cdot)}{\|w^k\|_{\tilde{L}^2(\p B_{r_k})}\|w^k_{r_k}-\overline{\Pr}(w^k_{r_k},1)\|_{\tilde{L}^2(\p B_1)}} \text{ in } B_1^\pm,\\
\p_{n+1}^+v_k + \p_{n+1}^-v_k&=0 \text{ on } B_1'\cap \{x_n>t\},\\
v_k&=0 \text{ on } B_1' \cap \{x_n< -t\}.
\end{split}
\end{equation}
Here $G^i_k$ is as in \eqref{eq:G_i} and we
have exploited the fact that $\Delta \overline{\Pr}(w_{r_k}^k, 1)=0$ in $\R^{n+1}_\pm$.
Using \eqref{eq:G_i2}, $\|w^k\|_{\tilde{L}^2(\p B_{r_k})}\geq r_k^{3/2+\epsilon_0}$ with $\epsilon_0\leq \frac{1}{2}(\alpha-\frac{1}{2})$ and the choice of $r_k$ in \eqref{eq:negative_contra3},
we infer that the inhomogeneity vanishes in the limit as $k\rightarrow \infty$.
Passing to the limit $k\rightarrow \infty$ hence entails \eqref{eq:linear}.\\

In the next step we will show that the limiting function $v_0$ is an element
of $\mathcal{E}\setminus\{0\}$.
However this is impossible, since at each step we have subtracted the projection of
$v_k$ onto $\mathcal{E}$.
This therefore completes the contradiction argument. \\

\emph{Step 2:} We show that $v_0\in \mathcal{E}\setminus\{0\}$. \\
\emph{Step 2a:} First we derive that for any $r\in (0,1/2)$,
\begin{align*}
\|v_0-\overline{\Pr}(v_0,r)\|_{\tilde{L}^2(\p B_r)}\leq C r^{3/2-\epsilon}
\end{align*}
for some small $\epsilon\in (0,1/2)$ fixed.
To show this, it suffices to deduce that for any $r\in (0,1/2)$ and for sufficiently large values of $k$ (which might depend on $r$),
\begin{align}\label{eq:growth_2}
\frac{\|w^k_{r_k}-\overline{\Pr}(w^k_{r_k},r)\|_{\tilde{L}^2(\p B_{r})}}{\|w^k_{r_k}-\overline{\Pr}(w^k_{r_k},1)\|_{L^2(\p B_1)}}\leq C r^{3/2-\epsilon}.
\end{align}
We distinguish two cases:\\
\emph{Case a:} If $r_kr$ is non-degenerate in the sense that
$\|w^k\|_{\tilde{L}^2(\p B_{r_kr})}\geq (r_kr)^{3/2+\epsilon_0}$, (due to the choice of $r_k$ in \eqref{eq:negative_contra3}) we have
\begin{align}\label{eq:growth_3}
\frac{\|w^k_{r_k}-\overline{\Pr}(w^k_{r_k},r)\|_{\tilde{L}^2(\p B_{r})}}{\|w^k_{r_k}-\overline{\Pr}(w^k_{r_k},1)\|_{L^2(\p B_1)}}=\frac{\|w^k-\overline{\Pr}(w^k,rr_k)\|_{\tilde{L}^2(\p B_{rr_k})}}{\|w^k-\overline{\Pr}(w^k,r_k)\|_{L^2(\p B_{r_k})}}\leq \frac{h_k(r_kr)}{h_k(r_k)}.
\end{align}
With this at hand, using the almost optimal regularity
$\|w^k_{r_k}\|_{\tilde{L}^2(\p B_{r})}\leq C_\epsilon r^{3/2-\epsilon}$, it is not hard
to see that $$\frac{h_k(r_kr)}{h_k(r_k)}\leq C_0r^{3/2-\epsilon}.$$
This entails \eqref{eq:growth_2}. \\
\emph{Case b:} If $r_kr$ is degenerate, i.e. if
$\|w^k\|_{\tilde{L}^2(\p B_{r_kr})}<(r_kr)^{3/2+\epsilon_0}$, we estimate
\begin{align*}
\|v_k-\overline{\Pr}(v_k,r)\|_{\tilde{L}^2(\p B_r)}=\frac{\|w^k-\overline{\Pr}(w^k,rr_k)\|_{\tilde{L}^2(\p B_{rr_k})}}{\|w^k-\overline{\Pr}(w^k,r_k)\|_{L^2(\p B_{r_k})}}\leq \frac{(r_kr)^{3/2+\epsilon_0}}{C_kr_k^{\epsilon_0}h_k(r_k)}.
\end{align*}
Using that $h_k(r_k)\geq  r_k^{3/2}$ yields
\begin{align*}
\|v_k-\overline{\Pr}(v_k,r)\|_{\tilde{L}^2(\p B_r)}
\leq \frac{(r_kr)^{3/2+\epsilon_0}}{C_kr_k^{\epsilon_0}r_k^{3/2}}
\leq  C_k^{-1}r^{3/2+\epsilon_0}.
\end{align*}

\emph{Step 2b:} We determine the leading order asymptotic profile of $v_0$ and show
that it is two-dimensional.
By Step 3a and by Lemma \ref{lem:class} (c.f. also Remark \ref{rmk:class}), which characterizes solutions to the
linear problem (\ref{eq:linear}),
\begin{align*}
v_0(x)=\sum_{i=1}^{n-1}c_ix_i p_{1/2}(x_n,x_{n+1})+ c_n p_{3/2}(x_n,x_{n+1})+O(|x|^2).
\end{align*}
We seek to show that $c_i=0$ for $i\in \{1,\cdots, n-1\}$. To this end we make use of minimality. Let $p^{\nu_k}_{3/2}(x)=c_k p_{3/2}(x'\cdot\nu_k, x_{n+1})$, where $\R^n\ni\nu_k\rightarrow e_n\in \R^n$ as $k\rightarrow \infty$. Moreover, after a rotation of coordinates, we can always assume that for each $k$, $\overline{\Pr}(w^k_{r_k},1)=d_k p_{3/2}^{e_n}$. Since
\begin{align*}
\|v_k-c_kp^{e_n}_{3/2}\|_{L^2(\p B_1)}\leq \|v_k-c_kp^{\nu_k}_{3/2}\|_{L^2(\p B_1)},
\end{align*}
it follows that
\begin{align*}
\int_{\p B_1} -2v_k (p^{e_n}_{3/2}-p^{\nu_k}_{3/2}) \dH
\leq \int_{\p B_1}(p^{\nu_k}_{3/2})^2 - (p^{e_n}_{3/2})^2 \dH =0.
\end{align*}
Let $\nu_k=e_i\sin\theta_k+e_n\cos\theta_k$, $i\in \{1,\cdots, n-1\}$,
for some $\theta_k \in (0,1)$ and $\theta_k\rightarrow 0_+$ as $k\rightarrow \infty$.
For such $\nu_k$ we have $(p^{\nu_k}-p^{e_n})/\theta_k\rightarrow \frac{3}{2}x_ip_{1/2}$.
Thus letting $k\rightarrow \infty $ and using the orthogonality of the expansion from Lemma \ref{lem:class}, we obtain that
\begin{align*}
\int_{\p B_1} 3c_i (x_ip_{1/2})^2 \dH \leq 0, \quad i\in \{1,\cdots, n-1\}.
\end{align*}
This implies that $c_i\leq 0$. By considering $\theta_k<0$ and $\theta_k\rightarrow 0_-$,
similar arguments lead to $c_i\geq 0$. Thus $c_i=0$ for $i\in \{1,\cdots, n-1\}$. \\

\emph{Step 2c:} We show that $v_0\in \mathcal{E}\setminus\{0\}$.
On the one hand Steps 3a and 3b
yield that $\|\nabla v_0\|_{L^2(B_1)}^2\geq \frac{3}{2}$.
Indeed, by Lemma \ref{lem:class} (and by the orthogonality of the decomposition)
it suffices to compute $\|\nabla q_k\|_{L^2(B_1)}^2$ for all $k$ individually. To this end,
we note that assuming that if $q_k$ is $L^2(\partial B_1)$ normalized, it holds that
\begin{align*}
 N(1,q_k) = \|\nabla q_k\|_{L^2(B_1)}^2 =k/2,
\end{align*}
where $k/2$ is the homogeneity of $q_k$ and $N(r,\cdot)$ denotes the
frequency function (with the same definition as for the constant coefficient thin obstacle problem).
This follows from the observation that the relevant frequency
function identities remain valid for solutions to the linear problem (\ref{eq:problem}) (if these are $C^{1}$ regular, which is ensured in our situation by the result from Step 2a).
By the normalization of
$v_0$ and the representation from Step 2b this then however implies that
\begin{align*}
 \|\nabla v_0\|_{L^2(B_1)}^2
 = \sum\limits_{k=3}^{\infty} |a_k|^2 \|\nabla q_k\|_{L^2(B_1)}^2
 \geq 3/2 \sum\limits_{k=3}^{\infty} |a_k|^2  = 3/2.
\end{align*}
On the other hand,
since $v_k\rightarrow v_0$ weakly in $W^{1,2}(B_1)$, we obtain
$\|\nabla v_0\|_{L^2(B_1)}^2\leq \liminf_{k\rightarrow \infty}\|\nabla v_k\|_{L^2(B_1)}^2
\leq \frac{3}{2}$. Thus, $\|\nabla v_0\|_{L^2(B_1)}^2=\frac{3}{2}$. This,
together with $\|v_0\|_{L^2(\p B_1)}=1$
and Step 3b, gives that $v_0(x)=c_np_{3/2}(x_n,x_{n+1})\in \mathcal{E}\setminus\{0\}$.
\end{proof}

The decay of our solutions towards $\mathcal{E}$ further implies control on the projections onto $\mathcal{E}$:

\begin{cor}\label{cor:negative}
Under the assumptions of Lemma~\ref{lem:negative} we have
\begin{align*}
\|\overline{\Pr}(w,s)-\overline{\Pr}(w,r)\|_{\tilde{L}^2(B_r)}\leq C r^{3/2+\epsilon_0/2},\quad 0<s<r<1/2.
\end{align*}
\end{cor}
\begin{proof}
We first claim that under the assumption of Lemma~\ref{lem:negative} it holds
\begin{align}\label{eq:solid}
\|w-\overline{\Pr}(w,r)\|_{\tilde{L}^2(B_r)}\leq C_0 r^{3/2+\epsilon_0/2} \mbox{ for all } r\in (0,1/2).
\end{align}
Indeed, by the almost optimal regularity of $w$, we have $\|w-\overline{\Pr}(w,r)\|_{\tilde{L}^2(\p B_r)}\leq C_0 r^{3/2+\epsilon_0/2}$. In order to obtain the estimate in the full ball $B_r$,
we apply Poincar\'{e}'s inequality (or the fundamental theorem of calculus) to the function $u(x):=w(x)-\overline{\Pr}(w,r)(x)$, i.e.
\begin{align*}
\|u\|_{L^2(B_r)}\leq C r\|\nabla u\|_{L^2(B_r)}+ C r^{1/2}\|u\|_{L^2(\p B_r)},
\end{align*}
and use that $\|\nabla (w-\overline{\Pr}(w,r))\|_{\LL(B_r)}\lesssim r^{-1}\|w-\overline{\Pr}(w,r)\|_{\LL(\p B_r)} + r^{1/2 + \beta}$ (c.f. \eqref{eq:proj} and $W(r,w)\leq C_n r^\beta$
from Corollary~\ref{cor:perturb_epi}). This gives \eqref{eq:solid}.

Next we show that for all $k,m\in \N$, $k\geq m\geq 1$,
\begin{align}
\label{eq:dyadic}
\|\overline{\Pr}(w,2^{-k})-\overline{\Pr}(w,2^{-m})\|_{\tilde{L}^2(B_{2^{-m}})}
\leq C (2^{-m})^{3/2+\epsilon_0/2}.
\end{align}
First if $k=m+1$, then by the homogeneity of functions in $\mathcal{E}$ and
by the triangle inequality we infer
\begin{align*}
&\|\overline{\Pr}(w,2^{-m-1})-\overline{\Pr}(w,2^{-m})\|_{\tilde{L}^2( B_{2^{-m}})}\\
&=2^{3/2}\|\overline{\Pr}(w,2^{-m-1})-\overline{\Pr}(w,2^{-m})\|_{\tilde{L}^2( B_{2^{-m-1}})}\\
&\leq 2^{3/2}\|w-\overline{\Pr}(w,2^{-m-1})\|_{\tilde{L}^2(B_{2^{-m-1}})}+2^{3/2}\|w-\overline{\Pr}(w,2^{-m})\|_{\tilde{L}^2(B_{2^{-m-1}})}\\
&\leq 2^{3/2}\|w-\overline{\Pr}(w,2^{-m-1})\|_{\tilde{L}^2(B_{2^{-m-1}})}+2^{3/2+(n+1)/2}\|w-\overline{\Pr}(w,2^{-m})\|_{\tilde{L}^2(B_{2^{-m}})},
\end{align*}
where in the last line we used $\|w-\overline{\Pr}(w,2^{-m})\|_{\tilde{L}^2(B_{2^{-m-1}})}\leq 2^{(n+1)/2}\|w-\overline{\Pr}(w,2^{-m})\|_{\tilde{L}^2(B_{2^{-m}})}$. Invoking the decay rate of $\|w-\overline{\Pr}(w,2^{-m})\|_{\tilde{L}^2(B_{2^{-m}})}$ from \eqref{eq:solid}, we have
\begin{align*}
\|\overline{\Pr}(w,2^{-m-1})-\overline{\Pr}(w,2^{-m})\|_{\tilde{L}^2( B_{2^{-m}})}\leq c_n(2^{-m})^{3/2+\epsilon_0/2}.
\end{align*}
This yields the claim of (\ref{eq:dyadic}) for $k=m+1$.\\
For $k>m+1$ we rely on a telescope sum argument:
\begin{align*}
&\|\overline{\Pr}(w,2^{-k})-\overline{\Pr}(w,2^{-m})\|_{\tilde{L}^2(B_{2^{-m}})}\\
&\leq \sum_{\ell=m}^{k-1}(2^{3/2})^{\ell-m+1}\|\overline{\Pr}(w,2^{-\ell-1})-\overline{\Pr}(w,2^{-\ell})\|_{\tilde{L}^2(B_{2^{-\ell}})}\\
&\leq C_n\sum_{\ell=m}^{k-1}(2^{-m})^{3/2}(2^{-\ell})^{\epsilon_0/2}\leq C_n(2^{-m})^{3/2+\epsilon_0/2}.
\end{align*}
Finally, by combining (\ref{eq:solid}) with (\ref{eq:dyadic}) and by applying the triangle
inequality, we conclude the proof of the lemma.
\end{proof}

Relying on the previous results (in particular on Lemma \ref{lem:negative}), we formulate the optimal growth behavior of solutions to the variable coefficient thin obstacle problem at free boundary points for $C^{0,\alpha}$ metrics in the case that $\alpha\in (1/2,1)$.

\begin{lem}[Optimal growth]
\label{lem:opt_growth}
Suppose that $w$ is a solution to the
variable coefficient thin obstacle problem~(\ref{eq:main2}) with $a^{ij}, g^i\in C^{0,\alpha}(B_1)$, $\alpha\in (1/2,1)$,
satisfying the normalization assumption (N). Moreover, assume that $\|w\|_{\LL(B_1)}= 1$.
Then there exists a constant $C>0$ depending on $n,\alpha$, $[a^{ij}]_{C^{0,\alpha}(B_1)}$
and $[g^i]_{C^{0,\alpha}(B_1)}$ such that for all $r\in(0,1/2)$ and for all
$x_0 \in \Gamma_w \cap B_{1/2}$
\begin{align*}
\|w\|_{\tilde{L}^2(B_r(x_0))} \leq C r^{3/2}.
\end{align*}
\end{lem}

\begin{proof}
We assume that $x_0=0\in \Gamma_w$ and only show the estimate for this point. For an arbitrary
free boundary point the argument is similar.\\
We apply Lemma \ref{lem:negative} in combination with Corollary
\ref{cor:negative} to infer
\begin{align*}
 \|w\|_{\LL( B_r)} &\leq \| w - \overline{\Pr}(w,r)  \|_{\LL( B_r)} + r^{3/2}
 \|\overline{\Pr}(w,r) - \overline{\Pr}(w,1)\|_{\LL( B_1)} \\
 & \quad
 + r^{3/2}\|\overline{\Pr}(w,1)\|_{\LL(B_1)} \leq C r^{3/2}.
\end{align*}
This yields the desired result.
\end{proof}

Last but not least, we discuss the situation, in which $\alpha\in(0,1/2)$:

\begin{lem}
 \label{lem:less_half}
Let $w:B_1 \rightarrow \R$ be a solution to the variable coefficient
thin obstacle problem,
for which the condition (N) is satisfied with $\alpha \in (0,1/2)$. Suppose further that
$\|w\|_{L^2(B_1)}=1$  and that $0\in \Gamma_w$.
Then there exists a constant $C_0>0$ depending on $\alpha, n, \| a^{ij}\|_{C^{0,\alpha}(B_1)},
\| g^{i}\|_{C^{0,\alpha}(B_1)}$ such that
\begin{align*}
\|w\|_{\tilde{L}^2(\p B_r)} \leq C_0 r^{1+ \alpha} \text{ for any } r\in (0,1/2).
\end{align*}
\end{lem}

\begin{proof}
We argue similarly as in Andersson \cite{An16} and in Lemma~\ref{lem:almost_linear}.
Arguing by contradiction, we can find a sequence of solutions $w^j$ with $0\in \Gamma_{w^j}$ and radii
$r_j$ such that
\begin{align*}
&\| w^j \|_{\LL(B_{r_j})} = j r_j^{1+\alpha },\\
&\| w^j \|_{\LL(B_{r})} \leq  j r^{1+ \alpha } \mbox{ for all }
r\in[r_j,1].
\end{align*}
We note that $r_j\rightarrow 0_+$, due to the uniform boundedness of
$\|w^j\|_{\tilde{L}^2(B_{r_j})}$. We consider the blow-up
\begin{align*}
 \tilde{v}_{j}(x):=\frac{w^j(r_jx)}{\|w^{j}\|_{\tilde{L}^2(B_{r_j})}}.
\end{align*}
By the choice of $r_j$ this satisfies the growth estimate
\begin{align}
\label{eq:growth_b}
 \|\tilde{v}_j\|_{\LL(B_{R})} =\frac{\|w^j\|_{\tilde{L}^2(B_{r_jR})}}{\|w^j\|_{\tilde{L}^2(B_{r_j})}}\leq R^{1+\alpha},\quad \forall R\in [1,\frac{1}{r_j}].
\end{align}
Compactness of $\tilde{v}_j$ and the vanishing of the inhomogeneity in the equation
follow as in
Lemma~\ref{lem:almost_linear} and Lemma
\ref{lem:blow_up}.
Hence, the limiting function
$\tilde{v}_0$ is a solution of the constant coefficient thin obstacle problem
\eqref{eq:limit_equ}.
Using the
growth estimate (\ref{eq:growth_b}) in combination with (ii) of Lemma \ref{lem:global}
implies that
$\tilde{v}_0(x)= a|x_{n+1}|$ with $a\leq 0$ (where we used that $\tilde{v}_{0}(x)=0$).
As the almost optimal regularity result of Proposition \ref{prop:Hoelder} however implies that $0\in\Gamma_{v_0}$, we deduce that $a=0$.
This contradicts the strong convergence
$\tilde{v}_j\rightarrow \tilde{v}_0$ in $L^2( B_1)$ and the normalization
$\|\tilde{v}_j\|_{\tilde{L}^2( B_1)}=1$.
\end{proof}

\begin{rmk}[$\alpha = 1/2$]
\label{rmk:3/2}
We emphasize that the proof of Lemma \ref{lem:less_half} does not hold uniformly as $\alpha \rightarrow \frac{1}{2}$. However, for the case $\alpha=1/2$ (a minor extension of) the technique from above allows us to prove that
\begin{align}
\label{eq:log}
\|w\|_{\LL(\p B_r)} \leq C r^{3/2}\ln(r) \mbox{ for any } r\in(0,1/2).
\end{align}
More generally, it allows us to prove
\begin{align*}
\|w\|_{\LL(\p B_r)} \leq C r^{3/2}f(r) \mbox{ for any } r\in(0,1/2).
\end{align*}
where $f(r)$ is such that $\frac{f(rR)}{f(r)}\leq f(R)$. Furthermore, considering the two-dimensional solution $w(r,\theta)=-(\ln r) r^{3/2}\cos(3\theta/2)$, we note that the logarithmic loss in (\ref{eq:log}) is indeed optimal.
\end{rmk}

We finally combine all the auxiliary results from this section (Lemma~\ref{lem:opt_growth} and Lemma~\ref{lem:less_half}) by formulating the optimal regularity result:

\begin{thm}[Optimal Hölder regularity]
\label{thm:Hoelder_opt}
Let $w$ be a solution to the variable coefficient thin obstacle problem in $B_1$ with
$a^{ij}, g^i\in C^{0,\alpha}(B_1)$ with $\alpha \in (0,1/2)$ or $\alpha\in (1/2,1)$ satisfying the normalization condition (N).
Then, there exists a constant
$C=C(n,\alpha, \|a^{ij}\|_{C^{0,\alpha}}, \|g^i\|_{C^{0,\alpha}})$ such that
\begin{align*}
\|w\|_{C^{1,\min\{\alpha,1/2\}}(B_{1/2})}\leq C \|w\|_{\LL(B_1)}.
\end{align*}
\end{thm}

The proof is a consequence of the free boundary growth estimates and the corresponding interior elliptic estimates away from the free boundary. Since it is very similar to Proposition~\ref{prop:Hoelder}, we do not repeat it here.

\section{Free Boundary Regularity}

\label{sec:freebdr}

In this last main section of the article, we harvest the results from Section \ref{sec:opt_reg} to derive regularity for (parts of) the free boundary (c.f. Theorem \ref{thm:freebound}). Here we analyze the part of the free boundary, which corresponds to the lowest possible homogeneity (modulo linear functions). This part of the free boundary corresponds to the \emph{regular} free boundary in the situation of the constant coefficient thin obstacle problem.\\

As a first result towards the free boundary regularity we investigate $L^2$ normalized blow-ups around free boundary points of the lowest possible vanishing order (modulo linear functions):

\begin{prop}\label{prop:vanishing_order}
Let $w$ be a solution to the thin obstacle problem \eqref{eq:vari_0}. Let $a^{ij}\in C^{0,\alpha}(B_1)$ and $g^i\in C^{0,\alpha}(B_1)$ for some $\alpha\in (1/2,1)$, and moreover suppose that they satisfy the normalization condition (N). Given $x_0\in \Gamma_w$, assume that
\begin{align}\label{eq:vanishing}
\kappa_{x_0}:=\liminf_{r\rightarrow 0_+}\frac{\ln( \|w-\ell_{x_0}\|_{\tilde{L}^2(B_r)})}{\ln r}<1+\alpha,
\end{align}
where $\ell_{x_0}(x)=a_{x_0}x_{n+1}$ with
\begin{align}
\label{eq:ax0}
a_{x_0}:=\lim_{x\in \Omega_w, x\rightarrow x_0}\p_{n+1}w(x).
\end{align}
Then there exists a subsequence $r_j\rightarrow 0_+$ and $p_{3/2}\in \mathcal{E}$ with $\|p_{3/2}\|_{\tilde{L}^2(B_1)}=1$, such that
\begin{align*}
\tilde{w}_{r_j}\rightarrow p_{3/2} \text{ in } C^{1,\beta}_{loc}(\R^{n+1}_\pm) \mbox{ for } \tilde{w}_{r_j}(x):=\frac{(w-\ell_{x_0})(r_jx)}{\|w-\ell_{x_0}\|_{\tilde{L}^2(B_{r_j})}}.
\end{align*}
Here $0<\beta <1/2$.
\end{prop}

\begin{proof}
The proof is very similar as the ones of Lemma~\ref{lem:almost_linear} and Lemma~\ref{lem:less_half}. We only show the claim for $x_0=0$. After subtracting $\ell_0(x)$, we may assume that $\ell_0(x)=0$. We note that the Signorini boundary condition on $B'_1$ does not change after subtracting $\ell_0$.\\
 Let $\tilde{\alpha}:=1+\alpha$. By \eqref{eq:vanishing} and Proposition~\ref{prop:Hoelder}, there exists a sequence $r_j\rightarrow 0$ such that $\|w\|_{\tilde{L}^2( B_{r_j})}=jr_j^{\tilde{\alpha}}$. We may also choose $r_j$ maximal, i.e. such that $\|w\|_{\tilde{L}^2( B_{r})}\leq jr^{\tilde{\alpha}}$ for $r \in [ r_j,1)$. We consider the blow-ups $\tilde{w}_{r_j}$. As in Lemma~\ref{lem:blow_up} we use our normalization and the maximality of $r_j$, in order to deduce that $\|\tilde{w}_{r_j}\|_{\tilde{L}^2(B_1)}=1$ and $\|\tilde{w}_{r_j}\|_{L^2(B_2)}\leq C$, where $C$ is uniform in $r_j$. Thus, $\|\tilde{w}_{r_j}\|_{W^{1,2}(B_1)}\leq C\|\tilde{w}_{r_j}\|_{L^2(B_2)}\leq C$. Similar estimates hold in the $C^{1,\beta}$ Hölder topologies.\\
Hence, up to a subsequence $\tilde{w}_{r_j}\rightarrow w_0$ in $C^{1,\beta}_{loc}(\R^{n+1}_{\pm})$ for $\beta\in (0,1/2)$, where $w_0$ is a solution to the thin obstacle problem for the Laplacian with zero obstacle and zero inhomogeneity. By the maximal choice of $r_j$, we further obtain
$$\sup_{B_R}|w_0|\leq C R^{\tilde{\alpha}} \mbox{ for all } R>1.$$ Since
$\tilde{\alpha}\in (1, 2)$, the Liouville theorem for the constant coefficient thin obstacle problem (c.f. Lemma \ref{lem:global} (iii)) yields that $w_0\in \mathcal{E}$. By our normalization and the strong convergences we infer that $\|w_0\|_{\tilde{L}^2(B_1)}=1$.
\end{proof}

Relying on Lemma \ref{lem:negative} and Corollary \ref{cor:negative}, we show that it is also possible to consider $3/2$-homogeneous blow-ups around free boundary points:

\begin{prop}
 \label{prop:32}
 Let $w:B_1 \rightarrow \R$ be a solution to the variable coefficient thin obstacle problem,
 satisfying the normalization condition (N). Assume that $0\in \Gamma_w$ and that $[a^{ij}]_{{C}^{0,\alpha}}+[g^i]_{{C}^{0,\alpha}}\leq \delta_0$ for some small $\delta_0>0$. Moreover, assume that $|\nabla w(0)|=0$. Then the limit
\begin{align*}
 \lim\limits_{r \rightarrow 0}w_r(x):= \lim\limits_{r \rightarrow 0}\frac{w(r x)}{ r^{3/2}} =: w_0(x)
\end{align*}
exists with
respect to the $C^{1,1/2-\epsilon}$ (with $\epsilon\in(0,1/2)$ arbitrary),
the $L^2$ and the weak $W^{1,2}$ topologies, and moreover $w_0\in \mathcal{E}$. Furthermore, there exists $\delta>0$ sufficiently small such that if
 \begin{align}
  \|w-h_{3/2}\|_{W^{1,2}(B_1)} \leq \delta,
 \end{align}
where $h_{3/2}$ is the model solution from Definition~\ref{defi:proj_3/2}, then
\begin{align}
\label{eq:fullball}
 \|w_0-h_{3/2} \|_{\LL(B_1)}\leq \frac{1}{10} \|h_{3/2}\|_{\LL(B_1)}.
\end{align}
In particular, $w_0$ is nontrivial.
\end{prop}

\begin{proof}
We first note that by Corollary~\ref{cor:negative}, for all $0<s<r<1/2$,
\begin{align*}
 \left\|\overline{\Pr}(w,s)-\overline{\Pr}(w,r)\right\|_{\tilde{L}^2(B_1)}
 =\left\| \frac{\overline{\Pr}(w,s)}{r^{3/2}}
 - \frac{\overline{\Pr}(w,r)}{r^{3/2}} \right\|_{\LL( B_r)}
 \leq C r^{\epsilon_0/2}.
\end{align*}
Here, $C= C(\|a^{ij}\|_{C^{1,\alpha}}, \|g^i\|_{C^{1,\alpha}},n,\|w\|_{\LL(B_1)})$.
Thus, the functions $r\mapsto \overline{\Pr}(w,r)\in \tilde{L}^2(B_1)$ form a Cauchy
sequence with limit
$w_0\in \mathcal{E}$. Passing to the limit $s\rightarrow 0$ in the above inequality, we
hence obtain
$\|\overline{\Pr}(w,r)-w_0\|_{\tilde{L}^2(B_1)}
\leq C r^{\epsilon_0/2}$. This together with triangle inequality and \eqref{eq:solid} yields
\begin{align}\label{eq:decay_w}
 \|w_r-w_0\|_{\LL(B_1)}
\leq \|w_r - \overline{\Pr}(r,w)\|_{\LL(B_1)}  +
  \left\|\overline{\Pr}(r,w)-w_0\right\|_{\LL(B_1)}
  \leq C r^{\epsilon_0/2}.
\end{align}
Therefore the limit $\lim\limits_{r\rightarrow 0}w_r=w_0$ exists and $w_0\in \mathcal{E}$ (which might a priori be such that $w_0=0$).\\
Next, we prove \eqref{eq:fullball}. To this end, we note that by \eqref{eq:decay_w}
\begin{align*}
\|h_{3/2}-w_0\|_{\tilde{L}^2(B_1)}
&\leq \|w_r - h_{3/2}\|_{\LL(B_1)} + \|w_r - w_0\|_{\LL(B_1)} \\
&\leq r^{-\frac{3}{2}-\frac{n+1}{2}}\|w - h_{3/2}\|_{L^2(B_r)} + C r^{\epsilon_0/2}\\
&\leq r^{-\frac{3}{2}-\frac{n+1}{2}} \delta + C r^{\epsilon_0/2}.
\end{align*}
Choosing first $r\geq 0$ such that
$C r^{\epsilon_0/2}\leq \frac{1}{20}\|h_{3/2}\|_{\LL(B_1)}$,
and then choosing $\delta$ such that
$r^{-\frac{3}{2}-\frac{n+1}{2}}\delta\leq \frac{1}{20}\|h_{3/2}\|_{\LL(B_1)}$,
yields
\begin{align*}
\|h_{3/2}-w_0\|_{\tilde{L}^2(B_1)}
\leq  \frac{1}{10}\|h_{3/2}\|_{\LL(B_1)}.
\end{align*}
As a consequence, if $\delta$ is chosen sufficiently small,
\begin{align*}
\|w_0\|_{\tilde{L}^2(B_1)}
\geq \frac{1}{2}\|h_{3/2}\|_{\LL(B_1)} \geq  c_n > 0.
\end{align*}
\end{proof}

Finally, we address the free boundary regularity, for which we combine the results of Propositions \ref{prop:vanishing_order} and \ref{prop:32}:

\begin{thm}[Free boundary regularity]
 \label{thm:freebound}
Let $w:B_2 \rightarrow \R$ be a solution to the variable coefficient thin obstacle problem with $a^{ij},g^i\in C^{0,\alpha}$, $\alpha\in (1/2,1)$, satisfying the normalization condition (N). Let $0\in \Gamma_w$ and assume that $\kappa_0<1+\alpha$, where $\kappa_0$ is defined as in \eqref{eq:vanishing}.
Then there exist a radius $r_0>0$ and an exponent $\gamma\in (0,1)$ such that
$\Gamma_w\cap B_{r_0}$ is a $C^{1,\gamma}$ $(n-1)$-dimensional manifold.
\end{thm}

\begin{proof}
Without loss of generality we assume that $\p_{n+1}w(0)=0$. By Proposition~\ref{prop:vanishing_order} there exists $r_0\in (0,1/2)$ such that (after a rotation)
\begin{align*}
\left\|\frac{w(r_0\cdot)}{\|w\|_{\tilde{L}^2(B_{r_0})}}-h_{3/2}\right\|_{C^1(B_1^{\pm})}\leq \frac{\delta}{10},
\end{align*}
where $\delta>0$ is the constant from Proposition~\ref{prop:32}.
Let $v(x):=\frac{w(r_0\cdot)}{\|w\|_{\tilde{L}^2(B_{r_0})}}$.
For $x_0 \in \Gamma_w \cap B_{1/2}$ we further define
\begin{align*}
v_{x_0}(x):=v(A(x_0)^{1/2} x + x_0)-b_{x_0}x_{n+1},
\end{align*}
where
$b_{x_0}:= (a^{n+1,n+1}(x_0))^{1/2}
\lim\limits_{x\in \Omega_w, \ x\rightarrow x_0} \p_{n+1} v(x)$ (in particular, $b_0=0$ by the assumption
that $\p_{n+1}w(0)=0$).
For this family of functions we note that for $|x_0|\leq r_0:= \min\left\{ \frac{\delta^{2}}{64 (\|v\|_{{C}^{0,1/2}}+[a^{ij}]_{{C}^{0,\alpha}(B_1)})^2},\frac{1}{100}\right\}$ we have
\begin{align*}
\|v_{x_0} - h_{3/2}\|_{C^1(B_1^{\pm})}
&=\left\|v\left(A(x_0)^{1/2}\cdot+x_0\right) -b_{x_0}(\cdot)_{n+1}- h_{3/2}(\cdot)\right\|_{C^1(B_1^{\pm})}\\
& \leq \|v - h_{3/2}\|_{C^{1}(B_{1}^{\pm})}
+\left\|v(A(x_0)^{1/2}\cdot + x_0) - v\right\|_{C^{1}(B_{1}^{\pm})}
+|b_{x_0}|\\
&\leq\frac{\delta}{10} + 2[\nabla v]_{{C}^{0,1/2}(B_{1}^{\pm})}|x_0|^{1/2}+[a^{ij}]_{{C}^{0,\alpha}(B_1)}|x_0|^{\alpha}\\
& < \delta.
\end{align*}
As moreover by definition $|\nabla v_{x_0}(0)|=0$, Proposition \ref{prop:32} is applicable to $v_{x_0}$ and entails that the limit
\begin{align*}
 \mathcal{E} \ni p_{x_0}(x):=\lim\limits_{r\rightarrow 0} \bar{v}_{x_0,r}(x),
 \quad \bar{v}_{x_0,r}(x):= \frac{v_{x_0}(rx)}{r^{3/2}}
\end{align*}
exists for all $x_0\in \Gamma_w\cap B_{r_0}$. Here we have used the off-diagonal assumption (N)
of the metric $a^{ij}$ to conclude that $p_{x_0}\in \mathcal{E}$.
Moreover, Proposition \ref{prop:32} also implies that
\begin{align}
\label{eq:controlb}
 \|\bar{v}_{x_0,r}-p_{x_0}\|_{\LL(B_1)}\leq C  s^{\epsilon_0/2}
 \mbox{ for } r\in(0,1),\ |x_0|\leq r_0.
\end{align}
Let $\tilde{p}_{x_0,r}:= \overline{\Pr}(v_{x_0},r)$.
We note that for $r=4|x_0|$
\begin{align*}
& C r^{3/2+\epsilon_0/2}
\geq \| v_{x_0} -\tilde{p}_{x_0,2r}\|_{\tilde{L}^2(B_{2r}(0))}\\
&= \frac{1}{\det{A(x_0)}^{1/2}}\left\|v(\cdot) - \p_{n+1}v(x_0) (\cdot)_{n+1} -\tilde{p}_{x_0,2r}(L_{x_0}^{-1}(\cdot))\right\|_{\tilde{L}^2(L_{x_0}(B_{2r}(0))},
\end{align*}
where $L_{x_0}(x)=A(x_0)^{1/2}x+x_0$.
If $[a^{ij}]_{{C}^{0,\alpha}}$ is sufficiently small,
we have $B_r(0)\subset L_{x_0}(B_{2r}(0))$, since
$|A(x_0)-I|\leq [a^{ij}]_{{C}^{0,\alpha}}|x_0|^\alpha$. Thus,
\begin{align}
\label{eq:inequ_1}
\left\|v(\cdot) - \p_{n+1}v(x_0) (\cdot)_{n+1} -\tilde{p}_{x_0,2r}(L_{x_0}^{-1}(\cdot))\right\|_{\tilde{L}^2(B_{r}(0))}\leq C r^{3/2+\epsilon_0/2}.
\end{align}
Assume that $\tilde{p}_{x_0,2r}(y)= c(x_0, 2r)\Ree(y'\cdot \tilde{\nu}(x_0,2r)+iy_{n+1})^{3/2}$ and let $\nu(x_0,2r):=(A(x_0)^{1/2})^{-1}\tilde{\nu}(x_0,2r)/|(A(x_0)^{1/2})^{-1}\tilde{\nu}(x_0,2r)|$, then
\begin{align*}
p_{x_0,2r}(x)&:=\tilde{p}_{x_0,2r}(L^{-1}_{x_0}(x))\\
&=c(x_0,2r)\Ree\left( \frac{(x-x_0)\cdot \nu(x_0,2r)}{(\nu(x_0,2r)\cdot A(x_0)\nu(x_0,2r))^{1/2}}+ i \frac{x_{n+1}}{a^{n+1,n+1}(x_0)^{1/2}}\right)^{3/2}.
\end{align*}
 Thus, rewriting the inequality (\ref{eq:inequ_1}) in terms of $p_{x_0,2r}$,
 we have shown that for $r=4|x_0|$
 \begin{align}
 \label{eq:inequ_2}
 \left\|v(\cdot) - \p_{n+1}v(x_0) (\cdot)_{n+1} -p_{x_0,2r}(\cdot)\right\|_{\tilde{L}^2(B_{r}(0))}\leq C r^{3/2+\epsilon_0/2}.
 \end{align}
In particular, this is applicable at the origin
 \begin{align}
\label{eq:inequ_3}
 \left\|v(\cdot)  -p_{0,2r}(\cdot)\right\|_{\tilde{L}^2(B_{r}(0))}\leq C r^{3/2+\epsilon_0/2},
 \end{align}
where we used that $\p_{n+1}w(0)=0$.
Combining (\ref{eq:inequ_2}), (\ref{eq:inequ_3}) and using the triangle inequality,
we infer that
 \begin{align*}
 \left\| \p_{n+1}v(x_0) x_{n+1}+(p_{0,2r}(x)-p_{x_0,2r}(x))\right\|_{\tilde{L}^2(B_{r}(0))}\leq C r^{3/2+\epsilon_0/2}.
 \end{align*}
Using the explicit expression of $p_{x_0,2r}$ and that $c(x_0,2r)\geq c_n>0$
(which follows from Proposition~\ref{prop:32} and our normalization), we obtain
\begin{align*}
|\p_{n+1}v(x_0)|\leq C|x_0|^{1/2+\epsilon_0/2}, \quad |\nu(x_0,2r)- \nu(0,2r)|\leq C |x_0|^{\epsilon_0/2}.
\end{align*}
By \eqref{eq:controlb},
\begin{align*}
|\nu(x_0,2r)-\nu(x_0,0_+)|\leq C|x_0|^{\epsilon_0/2}, \quad |\nu(0,2r)-\nu(0,0_+)|\leq C|x_0|^{\epsilon_0/2}.
\end{align*}
Thus by the triangle inequality
\begin{align*}
|\nu(x_0,0_+)-\nu(0,0_+)|\leq C|x_0|^{\epsilon_0/2}.
\end{align*}
This however corresponds to the H\"older continuity of the normal
$\nu_{x_0}$ to $\Gamma_w \cap B_{r_0}'$, since $x_0$ with $|x_0|\leq r_0$
was arbitrary,
and hence concludes the argument.
\end{proof}

\begin{rmk}[Modifications for the \emph{boundary} thin obstacle problem]
\label{rmk:boundary}
In the \\
case of the \emph{boundary} thin obstacle problem (which is also called the Signorini problem) we can relax the regularity assumption on the coefficients $a^{ij}$, due to the even symmetry about $x_{n+1}$ (which implies that $\p_{n+1}w(x_0)=0$ for all $x_0\in \Gamma_w$). More precisely, going through our arguments one can show that, if $a^{ij}\in C^{0,\beta}$ for some $\beta\in (0,1)$ and if the inhomogeneity $g^i\in C^{0,\alpha}$ for some $\alpha\in (1/2,1)$, then the solution satisfies $\sup_{B_r(x_0)}|w|\leq Cr^{3/2}$ at each $x_0\in \Gamma_w\cap B_{1/2}$. \\
Indeed, to observe this, it suffices to note that in Section \ref{sec:almost_opt}, we do not have to subtract any linear order polynomial in order to infer the almost optimal regularity result from Proposition \ref{prop:Hoelder} because $|\nabla w|=0$ along the free boundary. This simplifies the argument substantially. It also entails that in the boundary thin obstacle analogue of (\ref{eq:G_i}) we directly infer that $a_0=0$ and may use that $|\nabla w|\leq C_{\epsilon} r^{1/2 - \epsilon}$ for any $\epsilon\in(0,1)$ to obtain the faster than $r^{1/2}$ decay rate of $G^i$ around the origin. This allows us to repeat the argument from Section \ref{sec:opt_reg} with the only $C^{0,\beta}$ regularity of $a^{ij}$. This results in the claimed improvements.\\
Moreover, if $a^{ij}\in C^{0,\beta}$ for some $\beta\in (0,1)$ and $g^i\in C^{0,\alpha}$ for some $\alpha\in (1/2,1)$, the free boundary of the \emph{boundary} thin obstacle problem is locally a $C^{1,\gamma}$ graph around each regular free boundary point $x_0$ (i.e. at points $x_0\in \Gamma_{w}\cap B_{1/2}'$ with $\kappa_{x_0}<1+\min\{\alpha, 1/2+\beta\}$).
\end{rmk}

\section*{Acknowledgements}
A.R. acknowledges a Junior Research Fellowship at Christ Church.
W.S. was partially supported by the Hausdorff Center for Mathematics and the Centre for Mathematics of the
University of Coimbra. Both authors would like to thank Herbert Koch for helpful discussions related to the project.

\section{Appendix}
\label{sec:app}
Last but not least, in this section we present a sketch of the proof of Lemma \ref{lem:class} following the
ideas of Andersson (c.f. \cite{An16}, Lemma 6.2).

\begin{proof}[Proof of Lemma \ref{lem:class}]
We argue in two steps and first discuss the behavior with respect to the tangential
variables $x_1,\cdots,x_{n-1}$,
which then, in a second step, allows us to reduce the problem to a problem in the two variables
$x_n, x_{n+1}$ only.\\

\emph{Step 1: Tangential variables.} We note that the countable set
of linearly independent eigenfunctions $\{l_k\}_{k\in \N}$ of the restriction of
(\ref{eq:problem}) onto the sphere $\partial B_1 =: S^n$ forms an orthonormal basis for $L^2(S^{n})$.
As a consequence any solution $w$ to (\ref{eq:problem}) can be decomposed into
\begin{align*}
 w(r,\theta) = \sum\limits_{k=1}^{\infty} \alpha_k(r) l_k(\theta),
\end{align*}
where $\theta \in S^n, r\in \R_+$. Orthogonality,
implies that $\alpha_k(r) = \tilde{\alpha}_k r^{\kappa(k)}$ with $\tilde{\alpha}_k\in \R$
being independent of $r$ and $\kappa(k)\in \R$.
Thus, without loss of generality, we may assume that
\begin{align*}
 w(x) = \sum\limits_{k=1}^{\infty} \tilde{a}_k q_k(x),
\end{align*}
where $q_k$ are homogeneous functions solving (\ref{eq:problem}) with some still
unknown homogeneity, which we seek to determine in the sequel. \\
Let hence $q$ be such a function with homogeneity $\kappa$. We differentiate $q$ $k$-times
in an arbitrary combination of the directions $e_1,\cdots,e_{n-1}$ and denote this by
$D^{\alpha} q$ for some multi-index $\alpha\in \N^{n+1}$ with $|\alpha|=k$, $\alpha_n =0=\alpha_{n+1}$.
Thus, choosing $k>\kappa$, we infer that $D^{\alpha} q$ is homogeneous of order $\kappa-k<0$.
By a difference quotient argument, regularity theory for the Dirichlet and Neumann problems and the
regularity of our domain, we obtain that $D^{\alpha}q\in W^{1,2}(B_1)\cap L^{\infty}(B_1)$.
Combining the boundedness with the negative homogeneity then implies that
$D^{\alpha}q = 0$.\\
Integrating this and using that $\alpha \in \N^{n+1}$ (with $|\alpha|=k$ and $\alpha_n =0=\alpha_{n+1}$)
was arbitrary in its first $(n-1)$ components, yields that $\kappa$ can be expressed as the sum of an integer and the
homogeneity of an arbitrary two-dimensional solution of (\ref{eq:problem}) in the variables
$x_n, x_{n+1}$. Hence it remains to investigate the two-dimensional
problem, in order to determine the full set of possible inhomogeneities.\\

\emph{Step 2: The two-dimensional problem.} A simple calculation in two-dimensional polar coordinates
shows that all solutions to (\ref{eq:problem}) are of the form
 \begin{align*}
v(x_n,x_{n+1})&= x_{n+1},\\
 v(x_n,x_{n+1}) &= \Ree(x_n + i|x_n|)^{\kappa}
 \mbox{ with } \kappa \in \left\{\frac{2n+1}{2}: n \in \N\right\},\\
 v(x_n,x_{n+1}) &= \Imm(x_n + i x_n)^{\kappa} \mbox{ with } \kappa \in \N.
\end{align*}
Combining this with the previous discussion concludes the proof.
\end{proof}

\begin{rmk}
 \label{rmk:class}
Going through the proof carefully, shows that it reveals more information than stated in Lemma \ref{lem:class}: The proof allows us to conclude that the only solution
of homogeneity $1/2$ is given by $\Ree(x_n + i|x_n|)^{1/2} $,
it shows that the solutions of homogeneity one are only linear polynomials
and it reveals the possible structure of solutions of homogeneity $3/2$ as being of the form
\begin{align*}
 v(x) = \Ree(x_n + i|x_n|)^{1/2}\sum\limits_{j=1}^{n-1}a_j x_j + b \Ree(x_n + i|x_n|)^{3/2}.
\end{align*}
We use this more detailed knowledge in Step 3b in the proof of Lemma \ref{lem:negative}.
\end{rmk}

\bibliography{citations}
\bibliographystyle{alpha}

\end{document}